\pdfoutput=1
\documentclass[11pt]{amsart}
\usepackage{amsmath,amsthm,verbatim,amssymb,amsfonts,amscd,graphicx}
\usepackage{graphics}
\usepackage{tensor}
\usepackage{enumitem}
\usepackage[utf8]{inputenc}
\usepackage{graphicx}
\usepackage{xcolor}
\usepackage{hyperref}
\usepackage{float}
\usepackage{harpoon}
\definecolor{brass}{rgb}{0.71, 0.65, 0.36}

\usepackage{tikz-cd}
\usepackage{slashed}

\theoremstyle{plain}

\newtheorem{theorem}{Theorem}[section]

\newtheorem{proposition}[theorem]{Proposition}

\theoremstyle{definition}

\theoremstyle{definition}

\def\Ric{\operatorname{Ric}}

\def\R{\mathbb{R}}

\def\sup{\operatorname{sup}}

\def\sup{\operatorname{sup}}

\def\div{\operatorname{div}}

\numberwithin{equation}{section}

\numberwithin{equation}{section}

\allowdisplaybreaks

\begin{document}

\title[Spacetime Harmonic Functions and Applications]{Spacetime Harmonic Functions and Applications to Mass}

\author[Bray]{Hubert Bray}
\author[Hirsch]{Sven Hirsch}
\author[Kazaras]{Demetre Kazaras}
\address{Department of Mathematics\\
Duke University\\
Durham, NC 27708, USA}
\email{bray@math.duke.edu, sven.hirsch@duke.edu, demetre.kazaras@duke.edu, yiyue.zhang@duke.edu}

\author[Khuri]{Marcus Khuri}
\address{Department of Mathematics\\
Stony Brook University\\
Stony Brook, NY 11794, USA}
\email{khuri@math.sunysb.edu}

\author[Zhang]{Yiyue Zhang}


\thanks{M. Khuri acknowledges the support of NSF Grant DMS-1708798, and Simons Foundation Fellowship 681443.}

\begin{abstract}
In the pioneering work of Stern \cite{Stern}, level sets of harmonic functions have been shown to be an effective tool in the study of scalar curvature in dimension 3. Generalizations of this idea, utilizing level sets of so called spacetime harmonic functions as well as other elliptic equations, are similarly effective in treating geometric inequalities involving the ADM mass. In this paper, we survey recent results in this context, focusing on applications of spacetime harmonic functions to the asymptotically flat and asymptotically hyperbolic
versions of the spacetime positive mass theorem, and additionally introduce a new concept of total mass valid in both settings which is encoded in interpolation regions between generic initial data and model geometries. Furthermore, a novel and elementary proof of the positive mass theorem with charge is presented, and the level set approach to the Penrose inequality given by Huisken and Ilmanen is related to the current developments. Lastly, we discuss several open problems.
\end{abstract}

\maketitle

\section{Introduction}
\label{sec1}
\setcounter{equation}{0}
\setcounter{section}{1}

Euclidean space and hyperbolic space possess the following extremality property: any compactly supported perturbation to these geometries must somewhere decrease their scalar curvature. The present work will survey and explain recent developments related to the 3-dimensional asymptotically flat and hyperbolic positive mass theorems in mathematical general relativity. Geometrically speaking, these theorems refine the aforementioned extremality properties of $\mathbb{R}^n$ and $\mathbb{H}^n$ into the nonnegativity of a geometric invariant of manifolds which are merely asymptotic to -- as opposed to identically equal to -- the model geometries outside of a compact set.

The fact that $\mathbb{R}^n$ admits no compact perturbations increasing its scalar curvature is known as the Geroch conjecture. By identifying faces of a large coordinate cube in $\mathbb{R}^n$, one can deduce this fact from the non-existence of positive scalar curvature (psc) metrics on the torus $T^n$, first observed by Schoen-Yau \cite{SY,SY1} in dimensions less than 8 and subsequently established by Gromov-Lawson \cite{GromovLawson1} in all dimensions. Though it is far from obvious, an argument due to Lohkamp \cite{Lo} shows that the more far-reaching (Riemannian) positive mass theorem also follows from the statement that the connected sum of a torus and a closed manifold admits no psc metric.  See Gromov's Four Lectures \cite[Section 3.3]{GromovFour} for  further discussion of this and other rigidity phenomena related to scalar curvature.

A novel approach to the study of scalar has been initiated by Stern \cite{Stern}, in which a central role is played by the level sets of harmonic functions. Analogies may be drawn between the use of such level sets and the application of stable minimal hypersurfaces instituted by Schoen and Yau \cite{SY,SY0}, while the harmonic functions themselves may be compared with the harmonic spinors found in the fundamental work of Gromov and Lawson \cite{GromovLawson1,GromovLawson2}. The level set technique may be expanded by replacing harmonic functions with solutions to other geometrically motivated elliptic equations, leading to new tools with which to investigate initial data sets for the Einstein equations. The main idea is as follows.

Consider an orientable compact 3-dimensional Riemannian manifold $(M,g)$ with boundary $\partial M$, and let $u\in C^{\infty}(M)$ be a function that, for simplicity of discussion, has no critical points and admits the boundary $\partial M$ as a level set. Then a modified Bochner identity combined with (two traces of) the Gauss equations for level sets of $u$, yields an expression that may be integrated by parts to produce
\begin{equation}\label{master}
\int_{M}\frac{\left(\Delta u\right)^2}{|\nabla u|}dV-\int_{\partial M}2H|\nabla u|dA=\int_M \left(\frac{|\nabla^2 u|^2}{|\nabla u|}+(R-2K)|\nabla u|\right)dV,
\end{equation}
where $R$ is the scalar curvature of $g$, $K$ is the Gauss curvature of level sets, and $H$ is the mean curvature of the boundary with respect to the outward pointing normal. This formula leads to several well-known results by choosing $\Delta u$ appropriately. In particular, the harmonic prescription $\Delta u=0$ is sufficient to show that the torus $T^3$ does not admit a metric of positive scalar curvature \cite{BrayStern,Stern}, and can be used to establish the Bray-Brendle-Neves rigidity of $S^1\times S^2$ \cite{BrayBrendleNeves,Stern}, as well as the Riemannian positive mass theorem \cite{BKKS}. Choosing $u$ to satisfy the spacetime harmonic function equation
\begin{equation}
\Delta u=-\left(\mathrm{Tr}_g k\right)|\nabla u|,
\end{equation}
where $k$ is a symmetric 2-tensor representing extrinsic curvature of a spacelike slice, gives rise to the spacetime version of the positive mass theorem in the asymptotically flat \cite{HKK} and asymptotically hyperbolic settings \cite{BHKKZ}. Furthermore, if $\mathcal{E}$ denotes a divergence free vector field on $M$, then the associated drift Laplacian produces an appropriate equation
\begin{equation}\label{chargedlaplace}
\Delta u=\langle \mathcal{E},\nabla u\rangle,
\end{equation}
for the positive mass theorem with charge. Lastly, the equation governing the level set formulation of inverse mean curvature flow
\begin{equation}
\Delta u=|\nabla u|^2
+\frac{\langle \nabla u,\nabla|\nabla u|\rangle}{|\nabla u|},
\end{equation}
together with a weighted version of \eqref{master} yields the Penrose inequality as proven by Huisken and Ilmanen \cite{HI}.

The goal of this paper is to survey these results, with an emphasis on the applications of spacetime harmonic functions, and to introduce a new concept of mass via interpolation with model geometries. Furthermore, we discuss connections between spacetime harmonic functions and harmonic spinors, and point out relations with versions of the Jang equation. In addition, a novel and elementary proof of the positive mass theorem with charge is given utilizing \eqref{chargedlaplace}, and several open questions are discussed.

\section{Background}
\label{background}
\setcounter{equation}{0}
\setcounter{section}{2}

General relativity gives the best description of gravity known
to date.  From a mathematical point of view this theory can be
described as the study of Lorentzian 4-manifolds $(\mathbf{M}^{4},
\mathbf{g})$, where the metric $\mathbf{g}$ arises as a
solution to the Einstein field equations
\begin{equation}
\mathbf{Ric}-\frac{1}{2}\mathbf{R}\mathbf{g}=8\pi T,
\end{equation}
where $\mathbf{Ric}$ and $\mathbf{R}$ denote Ricci and scalar curvature, and $T$ is the stress-energy tensor which contains all relevant information concerning the matter fields.
It is often assumed that $T(v,w)\geq 0$ for all future pointing
non-spacelike vectors $v$, $w$, which simply asserts that all observed
energy densities are nonnegative, and is known as the \textit{dominant energy condition}.

It turns out that for many problems in general relativity,
including those addressed here, it is not necessary
to consider the entire spacetime, but rather just a spacelike
slice. Therefore we will restrict attention to initial data sets for
the Einstein equations $(M,g,k)$, consisting of a Riemannian 3-manifold $M$ with metric $g$ and a symmetric
2-tensor $k$ representing extrinsic curvature of the embedding into spacetime. As a consequence of this embedding, the tensors $g$ and $k$ must satisfy compatibility conditions, known as the \textit{constraint equations}, arising from traces of the Gauss-Codazzi relations
\begin{equation}
2\mu:=16\pi T(n,n)=R+(\mathrm{Tr}_{g} k)^{2}-|k|^{2},\text{ }\text{ }\text{ }\text{ }\text{ }\text{ }\text{ }
J:=8\pi
T(n,\cdot)=\operatorname{div}_{g}\left(k-(\mathrm{Tr}_{g} k)g\right),
\end{equation}
where $R$ is the scalar curvature of $g$, $n$ is the unit timelike normal to $M$, and $\mu$, $J$ are interpreted as the matter and momentum density of the matter fields. We note that the dominant energy condition implies that $\mu\geq |J|$, which places significant restrictions on the possible geometry and topology of an initial data set.  For instance, in the time symmetric ($k=0$) and maximal ($\mathrm{Tr}_{g}k=0$) cases this condition yields nonnegative scalar curvature.

An initial data set will be referred to as \textit{asymptotically flat}, if outside a compact set $M$ is the disjoint union of a finite number of ends, and for each end there is a diffeomorphism to the complement of a ball $\psi: M_{end}\rightarrow\mathbb{R}^3 \setminus B$ such that
\begin{equation}\label{AF}
\psi_{*} g-\delta \in C^{l,\alpha}_{-q}(\mathbb{R}^3 \setminus B),\quad\quad
\psi_* k \in C^{l-1,\alpha}_{-q-1}(\mathbb{R}^3 \setminus B),
\end{equation}
for some $l\geq 2$, $\alpha\in (0,1)$, and $q>\tfrac{1}{2}$. See \cite{Lee} for the definition of weighted H\"{o}lder spaces. The Hamiltonian formulation of general relativity given by Arnowitt, Deser, and Misner (ADM) \cite{ADM} gives rise to the total energy $E$ and linear momentum $P$ of each end. It will be assumed that the energy and momentum densities are integrable $\mu, J \in L^1(M)$, so that the ADM quantities are well-defined \cite{Bartnik,Chrusciel} and given by
\begin{equation}
E=\lim_{r\rightarrow\infty}\frac{1}{16\pi}\int_{S_{r}}\sum_i \left(g_{ij,i}-g_{ii,j}\right)\nu^j dA,\quad\quad
P_i=\lim_{r\rightarrow\infty}\frac{1}{8\pi}\int_{S_{r}} \left(k_{ij}-(\mathrm{Tr}_g k)g_{ij}\right)\nu^j dA,
\end{equation}
where $\nu$ is the unit outer normal to the coordinate sphere $S_r$ of radius $r=|x|$ and $dA$ denotes its area element. The ADM mass is then the Lorentzian length $m=\sqrt{E^2-|P|^2}$ of the ADM 4-momentum $(E,P)$. Under the dominant energy condition $\mu\geq |J|$, the 4-momentum is timelike unless the slice originates from Minkowski space, in which case it vanishes. This is the content of the spacetime positive mass theorem.

\begin{theorem}
Let $(M,g,k)$ be a 3-dimensional, complete, asymptotically flat initial data set for the Einstein equations. If the dominant energy condition is satisfied, then $E\geq |P|$ in each end. Moreover, $E=|P|$ in some end only if $E=|P|=0$ and the initial data arise from Minkowski space.
\end{theorem}

Initial versions of this result were obtained by Schoen and Yau \cite{SY1,SY2,SY3.0} as well as by Witten \cite{ParkerTaubes,Witten} in the early 1980's. The non-spacelike nature of the 4-momentum was established by Witten, with the hypersurface Dirac operator and a generalized Lichnerowicz formula, along with an outline for the rigidity statement. The case of equality was investigated further by Ashtekar and Horowitz \cite{AH}, and Yip \cite{Yp}, with a full proof given by Beig and Chru\'{s}ciel \cite{BC} under the asymptotic assumption $\mu, |J| =O(|x|^{-q-5/2})$. The Schoen and Yau approach treated the time-symmetric case first \cite{SY1}, when $k=0$, with stable minimal hypersurfaces, and then reduced the general case to this situation by employing the Jang equation \cite{SY2}. This technique proved that $E\geq 0$ along with the appropriate rigidity statement, assuming the additional asymptotic hypothesis $\mathrm{Tr}_g k=O(|x|^{-q-3/2})$ \cite{Eichmair2}. Alternatively, Eichmair, Huang, Lee, and Schoen \cite{EHLS} generalized the minimal hypersurface strategy by employing stable marginally outer trapped surfaces (MOTS) to show $|E|\geq |P|$. While the case of equality was not treated in \cite{EHLS}, Huang and Lee \cite{HuangLee} show that the rigidity statement follows from the inequality between $E$ and $P$. Two other strategies have been used for the Riemannian version of the positive mass theorem, namely a Ricci flow proof by Li \cite{Li} and the inverse mean curvature flow proof of Huisken and Ilmanen \cite{HI}.

In the higher dimensional setting, the MOTS approach \cite{EHLS} extends in a straightforward manner for dimensions up to and including 7. The Jang deformation has also been generalized by Eichmair \cite{Eichmair2} to these dimensions as well. Combining this with the rigidity argument of Huang and Lee \cite{HuangLee}, yields the desired result for dimensions less than 8 without the spin assumption. For spin manifolds Witten's method carries over to all dimensions, with the case of equality given by Chru\'{s}ciel and Maerten in \cite{CM}. In the case of K\"{a}hler manifolds, Hein and LeBrun \cite{HL} prove the Riemannian positive mass theorem in all dimensions. A compactification argument by Lohkamp \cite{Lo2}, analogous to the Riemannian case \cite{Lo}, reduces the inequality between $E$ and $P$ to the nonexistence of initial data with a strict dominant energy condition on the connected sum of a torus and a compact manifold. Moreover, in the articles \cite{SY3} and \cite{Lo1} the higher dimensional Riemannian problem is addressed by Schoen and Yau, and Lohkamp respectively. We recommend the book by Lee \cite{Lee}, for a detailed discussion of topics related to the positive mass theorem.

Asymptotically hyperbolic manifolds appear naturally in general relativity within two contexts. Namely, as asymptotically totaly geodesic spacelike hypersurfaces in asymptotically anti-de Sitter (AdS) spacetimes, as well as asymptotically hyperboloidal spacelike hypersurfaces in asymptotically flat spacetimes. Here we will focus on the second type, the quintessential example of which is the totally umbilic hyperboloid $t=\sqrt{1+r^2}$ in Minkowski space, whose induced metric $b=\tfrac{dr^2}{1+r^2}+r^2 \sigma$ is that of hyperbolic 3-space $\mathbb{H}^3$, where $\sigma$ is the round metric on $S^2$. An initial data set will be referred to as \textit{asymptotically hyperboloidal}, if outside a compact set $M$ is the disjoint union of a finite number of ends, and for each end there is a diffeomorphism to the complement of a ball $\psi: M_{end}\rightarrow\mathbb{H}^3 \setminus B$ such that
\begin{equation}
h:=\psi_{*} g-b \in C^{l,\alpha}_{-q}(\mathbb{H}^3 \setminus B),\quad
p:=\psi_* (k-g) \in C^{l-1,\alpha}_{-q}(\mathbb{H}^3 \setminus B),\quad
\psi_* \mu, \psi_* J\in C^{l-2,\alpha}_{-3-\epsilon}(\mathbb{H}^3 \setminus B)
\end{equation}
for some $l\geq 2$, $\alpha\in (0,1)$, $q>\tfrac{3}{2}$, and $\epsilon>0$. With these asymptotics, a well-defined notion of total energy and linear momentum is possible due to Chru\'{s}ciel, Herzlich, Jezierski, and {\L}\c{e}ski \cite{ChruscielHerzlich,CJL,Michel}. Namely, for each function $V\in C^{\infty}(\mathbb{H}^3)$ consider the \textit{mass functional}
\begin{equation}
H(V)=\lim_{r\rightarrow\infty}\frac{1}{16\pi}\int_{S_{r}}
\left[V\left(\operatorname{div}_b h -d\mathrm{Tr}_b h\right)+\mathrm{Tr}_b\left( h+2p\right)dV-\left(h+2p\right)\left(\nabla V,\cdot\right)\right](\nu) dA_b,
\end{equation}
where again $\nu$ is the unit outer normal to the coordinate sphere $S_r$. In order to isolate the 4-momentum, the function $V$ should satisfy the `static equation' $\mathrm{Hess}_b V=Vb$. A basis of solutions is obtained by restricting the coordinate functions of Minkowski space to the canonical hyperboloid: $\sqrt{1+r^2}$, and $x^i$, $i=1,2,3$. The energy, linear momentum, and mass are then given by
\begin{equation}
E=H\left(\sqrt{1+r^2}\right),\quad\quad P_i =H\left(x^i \right),\text{ }\text{ }i=1,2,3,\quad\quad m=\sqrt{E^2 -|P|^2}.
\end{equation}
In an asymptotically flat spacetime, the difference between the hyperboloidal mass and the ADM mass, is related to the amount of mass lost due to radiation.

The study of mass for asymptotically hyperbolic manifolds was initiated by Wang \cite{Wang}. The asymptotics utilized in \cite{Wang} are more restrictive, however they are more concrete in that they identify which portion of the metric and extrinsic curvature contribute to the mass. An asymptotically hyperboloidal initial data set is said to have \textit{Wang asymptotics} if $\tau=3$, and there are symmetric 2-tensors $\mathbf{m}$ and $\mathbf{p}$ on $S^2$ such that
\begin{equation}
\psi_* g= \frac{dr^2}{1+r^2} +r^2 \left(\sigma+\frac{\mathbf{m}}{r^3}+O_2(r^{-4})\right),\quad\quad
\psi_* (k-b)= \frac{\mathbf{p}}{r}+O_1(r^{-2}).
\end{equation}
The trace $\mathrm{Tr}_{\sigma}\left(\mathbf{m}+2\mathbf{p}\right)$ is typically referred to as the \textit{mass aspect function} and gives rise to the 4-momentum in this setting
\begin{equation}
E=\frac{1}{16\pi}\int_{S^2}\mathrm{Tr}_{\sigma}\left(\mathbf{m}+2\mathbf{p}\right)dA_{\sigma},
\quad\quad P_i =\frac{1}{16\pi}\int_{S^2}x^i \mathrm{Tr}_{\sigma}\left(\mathbf{m}+2\mathbf{p}\right)dA_{\sigma}, \text{ }\text{ }i=1,2,3,
\end{equation}
where $x^i$ are Cartesian coordinates of $\mathbb{R}^3$ restricted to the unit sphere $S^2$. We point out that the asymptotic conditions for the extrinsic curvature are integral to the definition of the energy and linear momentum. For instance, there exist asymptotically hyperboloidal slices of the Schwarzschild, as well as asymptotically totally geodesic slices of the AdS Schwarzschild, whose induced metrics are that of hyperbolic space \cite[Remark 1.5]{ChenWangYau}. The hyperboloidal version of the positive mass theorem may be stated as follows.

\begin{theorem}
Let $(M,g,k)$ be a 3-dimensional, complete, asymptotically hyperboloidal initial data set for the Einstein equations with a single end, and $l\geq 6$. If the dominant energy condition is satisfied, then $E\geq |P|$. Moreover, if $E=0$ and Wang asymptotics are satisfied, then the initial data arise from Minkowski space.
\end{theorem}

This theorem as stated is established by Sakovich in \cite{Sakovich}. Her method derives from a deformation argument of Schoen and Yau \cite{SY4}, which involves solving the Jang equation so that the solution admits hyperboloidal asymptotics. The induced metric on the Jang graph is then asymptotically flat, with ADM mass that agrees (up to a positive multiplicative factor) with the hyperbolic mass; further generalizations of this deformation procedure have been studied in \cite{ChaKhuriSakovich,ChaKhuri}. The desired result then follows from the positive mass theorem in the asymptotically flat setting.  The initial proofs by Wang \cite{Wang}, as well as Chru\'{s}ciel and Herzlich \cite{ChruscielHerzlich}, relied on spinor techniques and are valid in all dimensions $n\geq 3$ (for spin manifolds) under the scalar curvature lower bound $R\geq -n(n-1)$ which is equivalent to the dominant energy condition when $k=g$ (the umbilic case).
There is a substantial literature in which spinor techniques have been applied to the study of hyperbolic mass, namely \cite{CM,CrMTo,Maerten,XZ,Zhang,Zhang1}.
Andersson, Cai, and Galloway \cite{ACG} were able to remove the spin condition for dimensions $n\leq 7$ with the added assumption that the mass aspect function does not change sign. Their proof, which was inspired by the minimal hypersurface approach of Schoen and Yau in the asymptotically flat setting, relies on stable constant mean curvature hypersurfaces as well as a generalization of the Lohkamp deformation \cite{Lo}. Recently the spin condition has been removed in all dimensions for the umbilic case (as well as others) by Chru\'{s}ciel and Delay \cite{ChruscielDelay}, where the proof relies on the asymptotically flat positive mass theorem. Huang, Jang, and Martin \cite{HJM} treat the rigidity statement in this setting. We would also like to mention an alternate proof of Chen, Wang, and Yau \cite{ChenWangYau} based on the asymptotic limit of quasi-local mass, which holds for a certain type of asymptotics.

\section{Statement of Results}
\label{results}
\setcounter{equation}{0}
\setcounter{section}{3}

In this section we present the main results concerning the study of mass via level sets of solutions to elliptic equations. A novel feature of this approach is that explicit lower bounds for the mass are achieved without the need to assume nonnegative scalar curvature, or more generally the dominant energy condition. This is in contrast to the spinor technique, which although yields explicit lower bounds, requires the dominant energy condition in order to establish existence of the appropriate harmonic spinor.

Consider the master identity \eqref{master}. As discussed in the introduction, different choices for $\Delta u$ in this formula lead to different geometric inequalities for the mass. We will first examine the choice of a harmonic function, $\Delta u=0$. Recall that for each asymptotically flat end $M_{end}$ there is a so called \textit{exterior region} $M_{ext}$ containing $M_{end}$, that has minimal boundary, and which is diffeomorphic to $\mathbb{R}^3$ with a finite number of disjoint balls removed \cite[Lemma 4.1]{HI}. Such a region is desired when applying \eqref{master}, in order to avoid possible nonseparating sphere level sets of $u$ that can contribute adversely to the mass inequality.

\begin{theorem}\label{thm1}
Let $(M,g)$ be a smooth complete asymptotically flat Riemannian 3-manifold having mass $m$ in a chosen end, with exterior region $(M_{ext},g)$. If $u$ is a harmonic function on the exterior region which is asymptotic to an asymptotically flat coordinate function, and satisfies zero Neumann boundary conditions on $\partial M_{ext}$, then
\begin{equation}\label{intthm1}
m \geq \frac{1}{16\pi} \int_{M_{ext}}\left(\frac{|\nabla^2 u|^2}{|\nabla u|}+R |\nabla u|\right) dV.
\end{equation}
Consequently, if the scalar curvature $R\geq 0$ then $m\geq 0$. Moreover, $m=0$ if and only if $(M,g)$ is isometric to $(\mathbb{R}^3,\delta)$. In the special case that $M$ is diffeomorphic to $\mathbb{R}^3$, the integral of \eqref{intthm1} may be taken over all of $M$.
\end{theorem}

This theorem, which is established in \cite{BKKS}, may be generalized to the case in which $(M,g)$ has a boundary of nonpositive mean curvature, where the mean curvature is computed with respect to the normal pointing towards the asymptotic end in question. These surfaces, as below, are referred to as `trapped' and are connected with gravitational collapse \cite{Wa}. When such surfaces are present, and the scalar curvature is nonnegative, the proof of \eqref{intthm1} produces a strict inequality $m>0$. We mention also that an expression for the mass, related to Theorem \ref{thm1}, was obtained by Miao in \cite{Miao2}. Furthermore, a version of this result for manifolds with corners is given by Hirsch, Miao, and Tsang \cite{HMT}.

The use of Neumann boundary conditions ensures that certain boundary integrals vanish. However, an alternate approach is available in which the harmonic functions need not have boundary conditions prescribed. This may be accomplished with the Mantoulidis-Schoen neck construction \cite{MS}, whereby the boundary spheres of the exterior region are capped-off with 3-balls of nonnegative scalar curvature to produce a new manifold that is diffeomorphic to $\mathbb{R}^3$. The function $u$ may then be taken to be harmonic on the new manifold of trivial topology.

We now consider the spacetime setting. Let $\Sigma$ denote a 2-sided closed hypersurface in $M$ with null expansions given by $\theta_{\pm}=H\pm \mathrm{Tr}_{\Sigma}k$, in which $H$ is the mean curvature of $\Sigma$ with respect to $\upsilon$ the unit normal pointing towards a designated end. The surface $\Sigma$ may be viewed as embedded within spacetime, where the null expansions are then the mean curvatures in the null directions $\upsilon\pm n$; here $n$ is the future pointing timelike normal to the spacelike hypersurface $(M,g,k)$.
It follows that physically these quantities can be interpreted as measuring the rate of change of area of a shell of light emanating from the surface in the
outward future/past direction, and hence are indicators of the strength of the gravitational field. A strong gravitational field is associated with an \textit{outer or inner trapped} surface, that is when $\theta_{+}<0$ or $\theta_{-}<0$. Furthermore, $\Sigma$ is referred to as a \textit{marginally outer or inner trapped surface} (MOTS or MITS) if $\theta_+ =0$ or $\theta_- =0$.

In analogy with the time symmetric setting, it is important to control the topology of regular level sets for the relevant function $u$ appearing in \eqref{master}. This was previously achieved with the help of an exterior region, obtained from identifying the outermost minimal surface with respect to a particular end.
In the spacetime setting it is not known whether an appropriate exterior region, using the outermost MOTS/MITS, always exists. In its place we use the notion of a
\textit{generalized exterior region} associated with a designated end. More precisely, as shown in \cite[Proposition 2.1]{HKK}, for each end there exists a new initial data set $(M_{ext},g_{ext},k_{ext})$ with the following properties. Namely, it has a single end that agrees with the original, $M_{ext}$ is orientable with a boundary (possibly empty) comprised of MOTS and MITS, and satisfies the homology condition $H_2(M_{ext},\partial M_{ext};\mathbb{Z})=0$.

The appropriate choice of equation to use in the spacetime context, when applying the primary identity \eqref{master}, is given by
\begin{equation}\label{spacetimeh}
\Delta u+\left(\mathrm{Tr}_g k\right)|\nabla u|=0.
\end{equation}
Solutions of this equation are called \textit{spacetime harmonic functions}.
The left-hand side of \eqref{spacetimeh} arises as the trace along $M$ of the \textit{spacetime Hessian}
\begin{equation}
\bar{\nabla}_{ij}u=\nabla_{ij}u+k_{ij}|\nabla u|,
\end{equation}
which indicates some similarity with the hypersurface Dirac operator introduced by Witten \cite{Witten}. Further discussion of spacetime harmonic functions may be found in Section \ref{S: spacetime PMT}. We will say that a spacetime harmonic function $u$, on a generalized exterior region $M_{ext}$, is \textit{admissible} if it realizes constant Dirichlet boundary data together with $\partial_{\upsilon}u\leq (\geq) 0$ on each boundary component satisfying $\theta_+ =0$ ($\theta_- =0$), and there is at least one point on each boundary component where $|\nabla u|=0$. The existence of admissible spacetime harmonic functions that asymptote to a given linear function in the asymptotically flat end is established in \cite[Lemma 5.1]{HKK}.

\begin{theorem}\label{thm2}
Let $(M,g,k)$ be a smooth complete asymptotically flat 3-dimensional initial data set for the Einstein equations, with energy $E$ and linear momentum $P$ in a chosen asymptotic end $M_{end}$.
\begin{itemize}
\item [(i)]
If the dominant energy condition holds, then a generalized exterior region $M_{ext}$ exists which is associated with $M_{end}$, and satisfies the dominant energy condition. Let $\langle \vec{a},x\rangle =a_i x^i$ be a linear combination of asymptotically flat coordinates of the associated end, with $|\vec{a}|=1$. If $u$ is an admissible spacetime harmonic function on $M_{ext}$, asymptotic to this linear function, then
\begin{equation}\label{spinteq}
E+\langle \vec{a}, P\rangle\geq \frac{1}{16\pi} \int_{M_{ext}}\left(\frac{|\bar{\nabla}^2 u|^2}{|\nabla u|}
+2(\mu-|J|)|\nabla u|\right) dV.
\end{equation}
Consequently $E\geq |P|$. Moreover if $E=|P|$ then $E=|P|=0$, and the data $(M,g,k)$ arise from an isometric embedding into Minkowski space.

\item [(ii)]
Let $(M_{ext},g,k)$ be a generalized exterior region which does not necessarily satisfy the dominant energy condition. Then \eqref{spinteq} still
holds.

\item [(iii)]
In the special case that $M$ is diffeomorphic to $\mathbb{R}^3$, the integral of \eqref{spinteq} may be taken over $M$.
\end{itemize}
\end{theorem}

This theorem is established in \cite{HKK}. A version also holds where the generalized exterior region has weakly trapped boundary, instead of an apparent horizon boundary. This means that each boundary component satisfies $\theta_{+}\leq 0$ or $\theta_{-}\leq 0$. In this case, if in addition the dominant energy conditions is valid, then the conclusion is the strict inequality $E>|P|$. Each boundary component thus has a nontrivial contribution to the mass, and it would be of interest to more accurately determine this amount.

It turns out that spacetime harmonic functions are instrumental in the study of
mass in the asymptotically hyperbolic regime as well. As in the asymptotically flat case, these functions must grow linearly in the asymptotic end in order to `pluck out' the mass from the identity \eqref{master}. In order to find the appropriate model function to which the spacetime harmonic function should approach,
consider the linear function $\ell=-t +\langle \vec{a},x\rangle$ in Minkowski space. When restricted to the hyperboloid $t=\sqrt{1+r^2}$, this function satisfies the spacetime harmonic function equation \eqref{spacetimeh}, in fact its spacetime Hessian vanishes. Moreover, the level sets of $\ell$ intersected with the hyperboloid give horospheres in hyperbolic space $\mathbb{H}^3$, and $\nabla_{\mathbb{H}^3}\ell$ is a conformal Killing field on hyperbolic space. These properties suggest the asymptote
\begin{equation}
v_{\vec{a}}=-\sqrt{1+r^2}+\langle \vec{a},x\rangle,
\end{equation}
which is defined in any asymptotically hyperboloidal end. Note that the functions $x^i$ may be expressed in the asymptotically hyperboloidal coordinate system by using their polar form.

\begin{theorem}\label{thm3}
Let $(M,g,k)$ be a smooth complete asymptotically hyperboloidal 3-dimensional initial data set for the Einstein equations, having one end with energy $E$ and linear momentum $P$. Fix $\vec{a}\in\mathbb{R}^3$ with $|\vec{a}|=1$, and assume that the second integral homology group $H_2(M;\mathbb{Z})$ is trivial. Then there exists a spacetime harmonic function $u$ which is asymptotic to $v_{\vec{a}}$, and
\begin{equation}\label{spinteqh}
E+\langle \vec{a}, P\rangle\geq \frac{1}{16\pi} \int_{M}\left(\frac{|\bar{\nabla}^2 u|^2}{|\nabla u|}
+2(\mu-|J|)|\nabla u|\right) dV.
\end{equation}
Consequently, if the dominant energy condition holds then $E\geq |P|$. Moreover if $E=0$, then the data $(M,g,k)$ arise from an isometric embedding into Minkowski space. In the special case that $k=g$, if $E=|P|$ then $E=|P|=0$ and $(M,g)$ is isometric to hyperbolic 3-space.
\end{theorem}

This result is established in \cite{BHKKZ}.
The topological assumption $H_2(M;\mathbb{Z})=0$ should not be considered as necessary to obtain mass lower bound \eqref{spinteqh}. In the general setting where the homology of $M$ is nontrivial, one should pass to an auxiliary initial data set with trivial homology and an end that is isometric (as initial data) to the original $(M,g,k)$. This type of auxiliary data set is an enhanced version of the generalized exterior region used for Theorem \ref{thm2}. The arguments of Theorem \ref{thm3} can then be carried out on this secondary space to obtain the desired result in full generality. It should be pointed out that our argument in the case of equality, under the umbilic assumption $k=g$, does not require the Huang-Jang-Martin result \cite{HJM}, although their theorem does imply the desired conclusion. We also mention the recent result of Jang and Miao \cite{JangMiao} which provides an expression for hyperbolic mass computed via horospheres, in the umbilic case.

Consider now a Riemannian 3-manifold $(M,g)$ augmented by a smooth vector field $\mathcal{E}$ representing an electric field. The triple $(M,g,\mathcal{E})$ will be referred to as \textit{charged asymptotically flat initial data}, if $M$ has the topology of an exterior region with a single asymptotically flat end $M_{end}$ and a finite number of asymptotically cylindrical ends, and $\mathcal{E}\in C^{0,\alpha}_{-q-1}(M_{end})$ with $q>\frac{1}{2}$, $\alpha\in(0,1)$. The total charge is given by
\begin{equation}\label{chargedef}
Q=\lim_{r\rightarrow\infty}\frac{1}{4\pi}\int_{S_r}\langle\mathcal{E},\upsilon\rangle dA,
\end{equation}
where $\upsilon$ is the unit outer normal to coordinate spheres $S_r$ in the asymptotically flat end. Typically $\mathcal{E}$ will be taken to be divergence free, meaning that there is no charge density, and in this case the total charge is finite as it is a homological invariant.
The model charged asymptotically flat initial data are time slices of the Majumdar-Papapetrou spacetime $\left(\mathbb{R}\times\left(\mathbb{R}^3 \setminus \cup_{i=1}^I p_i\right),g_{MP}\right)$ where
\begin{equation}
g_{MP}=-\phi^{-2} dt^2 +\phi^2 \delta,\quad\quad \mathcal{E}_{MP}=\nabla\log \phi,\quad\quad
\phi=1+\sum_{i=1}^{I}\frac{q_i}{r_i},
\end{equation}
with $r_i$ the Euclidean distance to each point $p_i\in\mathbb{R}^3$. Each such point represents a degenerate black hole, in the sense that the horizon is not present within the initial data but rather lies at the bottom of the associated asymptotically cylindrical end. The constants $q_i >0$ give the charge and mass of each black hole, and the total charge as well as the total mass agrees with $\sum_{i=1}^{I}q_i$. In order to establish a version of the positive mass theorem with charge, we utilize \eqref{master} with functions $u$ satisfying the equation
\begin{equation}\label{fjguru}
\Delta u-\langle\mathcal{E},\nabla u\rangle=0.
\end{equation}
Solutions to this equation will be referred to as \textit{charged harmonic functions}, and the
associated \textit{charged Hessian} is given by
\begin{equation}
\hat{\nabla}_{ij} u=\nabla_{ij}u+\mathcal{E}_i u_j +\mathcal{E}_j u_i -\langle\mathcal{E},\nabla u\rangle g_{ij},
\end{equation}
where $u_i$ denote partial derivatives. Observe that the charged (drift) Laplacian in \eqref{fjguru} arises from a trace of the charged Hessian.

\begin{theorem}\label{thm4}
Let $(M,g,\mathcal{E})$ be a smooth complete charged asymptotically flat initial data set, having divergence free electric field, mass $m$, and charge $Q$. There exists a charged harmonic function $u$ on $M$ which is asymptotic to an asymptotically flat coordinate function in $M_{end}$, and remains bounded along asymptotically cylindrical ends, such that
\begin{equation}\label{chargedlower}
m-|Q|\geq\frac{1}{16\pi}\int_M\left(\frac{|\hat\nabla^2u|^2}{|\nabla u|}+(R-2|E|^2)|\nabla u|\right)dV.
\end{equation}
Consequently, if the charged dominant energy condition is satisfied $R\geq 2|E|^2$, then $m\geq |Q|$. Moreover, $m=|Q|$ if and only if $(M,g,\mathcal{E})$ is isometric (as charged initial data) to the time slice of a Majumdar-Papapetrou spacetime.
\end{theorem}

Physically, the charged dominant energy condition hypothesis may be interpreted as stating that the non-electromagnetic matter fields satisfy the dominant energy condition. The inequality $m\geq|Q|$, known as the positive mass theorem with charge, was first established by Gibbons, Hawking, Horowitz, and Perry \cite{GHHP} using spinorial techniques (see also \cite{BartnikChrusciel}). Their result allowed for the inclusion of extrinsic curvature $k$, but did not allow asymptotically cylindrical ends and thus could not treat the case of equality. Novel features of Theorem \ref{thm4} include the lower bound \eqref{chargedlower} for the difference $m-|Q|$, which does not rely on an energy condition, as well as a new approach to the rigidity statement which does not appear to be fully resolved in all cases \cite{CRT}. Other related results may be found in \cite[Theorem 2.1]{AKY}, \cite[Theorem 2]{Jaracz}, and \cite[Theorem 2]{KhuriWeinstein}.

Theorem \ref{thm4} will be established in Section \ref{S: charge} below. A discussion of the proofs of Theorems \ref{thm2} and \ref{thm3} are given in Sections \ref{S: spacetime PMT} and \ref{S: hyperbolic PMT}, respectively. The new concept of mass obtained through interpolation with model geometries is also given in Section \ref{S: hyperbolic PMT}.
In Section \ref{S: master} an outline of the proof of identity \eqref{master} is provided, and it is shown how the inverse mean curvature flow approach \cite{HI} to the Riemannian Penrose inequality may be placed within the context of the level set methods presented here.  Lastly, motivation for the spacetime harmonic function equation as well as connections to other equations are examined in Section \ref{S: spacetime harmonic functions}, while open questions are proposed in Section \ref{S: questions}.

\section{The Level Set Formula}
\label{S: master}
\setcounter{equation}{0}
\setcounter{section}{4}

The purpose of this section is to outline the proof of the main identity \eqref{master}. For simplicity, throughout this section it will be assumed that $|\nabla u|\neq 0$. This restriction is not necessary, however it will allow us to display the essence of the argument without including certain technical details. In order to obtain the general result, $|\nabla u|$ should be replaced with $\sqrt{|\nabla u|^2 +\varepsilon}$ in what follows, and then $\varepsilon$ taken to zero after an appropriate application of Sard's theorem and a Kato inequality. We refer the reader to \cite[Section 3]{HKK} for details.

\subsection{Primary identity}

First recall Bochner's identity
\begin{equation}
\frac12\Delta|\nabla u|^2=|\nabla^2u|^2+\Ric(\nabla u,\nabla u)+\langle \nabla u,\nabla \Delta u\rangle,
\end{equation}
and note that this implies
\begin{equation}
\Delta |\nabla u|=\frac1{|\nabla u|}\left(|\nabla^2u|^2+\Ric(\nabla u,\nabla u)+\langle \nabla u,\nabla \Delta u\rangle-|\nabla |\nabla u||^2\right).
\end{equation}
Consider now the $t$-level set $\Sigma_t$ of $u$. The unit normal to this surface is given by $\nu=\frac{\nabla u}{|\nabla u|}$, and the corresponding second fundamental form and mean curvature respectively take the form
\begin{equation}\label{Ah}
A_{ij}=\frac{\nabla_{ij}u}{|\nabla u|},\quad\quad\quad H=\frac1{|\nabla u|}\left(\Delta u-\nabla_{\nu\nu}u\right).
\end{equation}
These formulas imply that
\begin{equation}
|\nabla u|^2 \left(H^2-|A|^2 \right)=2|\nabla|\nabla u||^2-|\nabla^2 u|^2+\left(\Delta u \right)^2-2\Delta u\nabla_{\nu\nu}u.
\end{equation}
Combining this with two traces of the Gauss equations
\begin{equation}
\Ric(\nu,\nu)=\frac12(R-2K+H^2-|A|^2),
\end{equation}
produces
\begin{equation}
\Delta |\nabla u|=\frac1{2|\nabla u|}\left(|\nabla^2u|^2+\left(R-2K \right)|\nabla u|^2+2\langle \nabla u,\nabla\Delta u\rangle+\left(\Delta u\right)^2-2\Delta u\nabla_{\nu\nu}u \right).
\end{equation}
Next observe that
\begin{equation}
\frac{\langle\nabla u,\nabla\Delta u\rangle}{|\nabla u|}
-\frac{\Delta u}{|\nabla u|}\nabla_{\nu\nu}u
=\mathrm{div}\left(\Delta u \frac{\nabla u}{|\nabla u|}\right)
-\frac{\left(\Delta u\right)^2}{|\nabla u|},
\end{equation}
and therefore
\begin{equation}\label{7890}
\div\left(\nabla |\nabla u|-\Delta u\frac{\nabla u}{|\nabla u|}\right)=\frac1{2|\nabla u|}\left(|\nabla^2u|^2+\left(R-2K \right)|\nabla u|^2 -\left(\Delta u\right)^2 \right).
\end{equation}
Integrating this over the manifold $M$ whose boundary is a level set of $u$, yields the desired result
\begin{equation}
\int_{M}\frac{\left(\Delta u\right)^2}{|\nabla u|}dV-\int_{\partial M}2H|\nabla u|dA=\int_M \left(\frac{|\nabla^2 u|^2}{|\nabla u|}+(R-2K)|\nabla u|\right)dV,
\end{equation}
where the mean curvature in the boundary integral is with respect to the unit outer normal.
It is interesting to note the vague similarities between equation \eqref{7890}
and \cite[Lemma 3.2]{KhuriXie}, in particular when $u$ is harmonic.

\subsection{Relation to Geroch monotonicity}

The Penrose inequality \cite{Mars} is a conjectural relation between the mass and horizon area for initial data satisfying the dominant energy condition, and is motivated by heuristic physical arguments tying it to the grand cosmic censorship conjecture \cite{Penrose}. In the asymptotically flat time symmetric case $k=0$, assuming nonnegative scalar curvature, the relation states that
\begin{equation}
m\geq \sqrt{\frac{\mathcal{A}}{16\pi}}
\end{equation}
where $\mathcal{A}$ denotes the area of the outermost minimal surface with respect to a particular end. Moreover, equality is achieved only for time slices of the Schwarzschild spacetime. These statements were established by Huisken and Ilmanen \cite{HI} for a single black hole, and by Bray \cite{Bray1} for multiple black holes; Bray and Lee \cite{BrayLee} extended this to dimensions $n\leq 7$.  The proof by Huisken and Ilmanen relies on monotonicity of the Hawking mass along inverse mean curvature flow. A level set characterization of the flow was used to overcome singular behavior. A strategy to generalize these strategies to the non-time symmetric setting may be found in \cite{BrayKhuri1,BrayKhuri,HanKhuri}.

It is natural to expect some connection between the level set approach for the Penrose inequality, and those that are used to study the positive mass theorem. Indeed,
consider a smooth `weight function' $f=f(u)$ defined on the levels of $u$, and suppose that the maximum and minimum levels are $\min_Mu=T_1$ and $\max_Mu=T_2$, which are achieved on the level set boundary of $M$. Then multiplying \eqref{7890} by $f$ and integrating by parts produces
\begin{align}\label{weights}
\begin{split}
&-\int_{\Sigma_{T_2}}2H|\nabla u|fdA+\int_{\Sigma_{T_1}}2H|\nabla u|fdA\\
=&\int_{T_1}^{T_2}f\int_{\Sigma_t}\left[H^2+|A|^2+R+2\frac{|\nabla|_{\Sigma_t} \nu(u)|^2}{|\nabla u|^2}-2H\frac{\Delta u}{|\nabla u|} \right]dAdt\\
&-\int_{T_1}^{T_2}4\pi f\chi(\Sigma_t)dt-2\int_{T_1}^{T_2} f'\int_{\Sigma_t}H|\nabla u|dAdt,
\end{split}
\end{align}
where $\chi(\Sigma_t)$ is the Euler characteristic of the $t$-level set. This expression is obtained with the help of the coarea formula, the Gauss-Bonnet theorem, and \eqref{Ah}. Now choose $u$ to solve the level set formulation of inverse mean curvature flow equation, and set the weight function as follows
\begin{equation}
\Delta u=|\nabla u|^2
+\frac{\langle \nabla u,\nabla|\nabla u|\rangle}{|\nabla u|},\quad\quad\quad
f(t)=\frac{e^{(t-T_1)/2}}{16\pi}\sqrt{\frac{|\Sigma_{T_1}|}{16\pi}}
=\frac{1}{16\pi}\sqrt{\frac{|\Sigma_{t}|}{16\pi}}.
\end{equation}
Inserting this into \eqref{weights}, integrating by parts once more, and using $H=|\nabla u|$ yields
\begin{align}
\begin{split}
&m_H(\Sigma_{T_2})-m_H(\Sigma_{T_1})\\
=&\int_{T_1}^{T_2}\sqrt{\frac{|\Sigma_t|}{16\pi}}\left[\frac12-\frac14\chi(\Sigma_t)+\frac1{16\pi}\int_{\Sigma_t}\left(\frac{2|\nabla _{\Sigma_t} H|^2}{H^2}+R+|A|^2-\frac12H^2\right)dA\right]dt,
\end{split}
\end{align}
where
\begin{equation}
m_H(\Sigma_t)=\sqrt{\frac{|\Sigma_t|}{16\pi}}\left(1-\frac1{16\pi}\int_{\Sigma_t}H^2dA \right)
\end{equation}
is the Hawking mass. If the level sets remain connected, which is valid in exterior regions \cite[Lemma 4.2]{HI}, then $\chi(\Sigma_t)\leq 2$ and the Hawking mass is nondecreasing. This is Geroch monotonicity \cite{Geroch}, which leads to a proof of the Penrose inequality for a single black hole. A version of the above arguments holds for the Penrose inequality with charge \cite{HI,Jang,KWY,McCormick} using the charged Hawking mass \cite{DK}.

\section{Spacetime Harmonic Functions}
\label{S: spacetime harmonic functions}
\setcounter{equation}{0}
\setcounter{section}{5}

Consider an asymptotically flat 4-dimensional spacetime $(\mathbf{M}^4,\mathbf{g})$ having initial data set $(M,g,k)$. The positive mass theorem in the time-symmetric case, Theorem \ref{thm1}, was proven \cite{BKKS} using the level sets of asymptotically linear harmonic functions on $(M,g)$. We are led to the analogue in the spacetime setting by taking into account intuition from the case of equality. More precisely, when the mass vanishes the initial data arise from Minkowski space $\mathbb{M}^4$, and in the time symmetric case this hypersurface is a constant time slice with the relevant harmonic functions given by $a_i x^i$ where $a_i$, $i=1,2,3$ are constants; here $x^i$ are the standard cartesian coordinates parameterizing Minkowski space, with $i=1,2,3$ representing spatial indices and $i=0$ representing the time index. For nonconstant time slices, the most natural generalization is to utilize linear functions $\pmb{\ell}=a_0 x^0+a_i x^i$ of all the coordinates in $\mathbb{M}^4$. The level set foliation within the initial data $(M,g,k)$ is then the intersection of these hyperplanes with the slice. In other words, the relevant function is the restriction of the spacetime linear combination to $M$. The issue is then to determine a canonical equation induced on the slice that is satisfied by this function.

To find the appropriate equation, let $\nabla$ and $\pmb{\nabla}$ be the Levi-Civita connections of the slice and spacetime, respectively. Linear functions of the coordinates in Minkowski space have vanishing spacetime Hessian, and thus when restricted to $M$ they must lie in the kernel of the hypersurface spacetime Laplacian $\mathbf{\Delta}=g^{ij}\pmb{\nabla}_{ij}$. This suggests that in a general spacetime, we should  consider functions $\mathbf{u}\in C^{2}(\mathbf{M}^4)$ that satisfy
\begin{equation}\label{09876565}
0=\mathbf{\Delta}\mathbf{u}=g^{ij}\left(\nabla_{ij}\mathbf{u}-k_{ij} n(\mathbf{u})\right)
=\Delta\mathbf{u}-\left(\mathrm{Tr}_g k\right)n(\mathbf{u})\quad\text{ on }\quad M,
\end{equation}
where the unit timelike normal to the slice is denoted by $n$.
Observe that $\mathbf{\Delta}\mathbf{u}$ may be considered as the divergence of $\pmb{\nabla}\mathbf{u}|_M$, using the ambient connection $\pmb{\nabla}$ acting on sections of the restricted bundle $T\mathbf{M}^4 |_{M}$. This is similar to the type of harmonic spinor employed by Witten \cite{Witten}, where the hypersurface Dirac operator is defined by the induced connection from $\mathbf{M}^4$. The equation \eqref{09876565} is not sufficient, however, because it still depends on the spacetime nature of the function, instead of being solely determined by the restriction $u=\mathbf{u}|_{M}$. In order to rectify this problem, a choice for $n(\mathbf{u})$ must be made in terms of intrinsic quantities associated with the data. In this pursuit we are guided by the lower bound for mass that arises from the identity \eqref{master}. It turns out that the desired choice is for the spacetime gradient $\pmb{\nabla}\mathbf{u}$ to be null, more precisely $n(\mathbf{u})=-|\nabla u|$. Going back to the intuition from Minkowski space, we see that the level sets of the relevant function $u$ consist of the intersection of light cones with the slice. With these considerations, we are then motivated to define the notion of a \textit{spacetime harmonic function} to be $u\in C^{2}(M)$ solving the \textit{spacetime Laplace} equation
\begin{equation}
\bar{\Delta} u=\Delta u+\left(\mathrm{Tr}_g k\right)|\nabla u|=0,
\end{equation}
which arises from a trace of the \textit{spacetime Hessian}
\begin{equation}
\bar{\nabla}_{ij}u=\nabla_{ij}u+k_{ij}|\nabla u|.
\end{equation}

\subsection{Association with the Jang equation}
\label{jang}

Let us generalize the derivation of the spacetime Laplacian. Recall that at a certain point a choice was made for the normal derivative $n(\mathbf{u})$.  There are other natural choices that can be made besides dictating that $\pmb{\nabla}\mathbf{u}$ be null. For instance, one possibility is to take
\begin{equation}\label{hyui}
n(\mathbf{u})=-\sqrt{a+|\nabla u|^2}
\end{equation}
for some constant $a\in \mathbb{R}$, assuming that the quantity inside the square root is nonnegative. Then the \textit{generalized spacetime harmonic function} equation becomes
\begin{equation}\label{aldkjf}
\Delta u+\left(\mathrm{Tr}_g k\right)\sqrt{a+|\nabla u|^2}=0.
\end{equation}
To gain intuition concerning the level sets of solutions, consider again the example of initial data $(M,g,k)$ for Minkowski space. The linear function $\pmb{\ell}$, by virtue of having a vanishing Hessian, will satisfy the generalized spacetime harmonic function equation as long as the coefficients $a_i$ are chosen to satisfy \eqref{hyui}. Since
\begin{equation}
-a_0^2 +\sum_{i=1}^3 a_i^2 =|\pmb{\nabla}\pmb{\ell}|^2 =-n(\pmb{\ell})^2 +|\nabla \pmb{\ell}|^2,
\end{equation}
we find that \eqref{hyui} is satisfied whenever
\begin{equation}
a=a_0^2 -\sum_{i=1}^3 a_i^2.
\end{equation}
It follows that the hyperplanes $\pmb{\ell}=const$ are timelike when $a>0$, spacelike when $a<0$, and null when $a=0$. Level sets of the generalized spacetime harmonic function, arising from the restriction of $\pmb{\ell}$ to the slice, are obtained by intersecting $M$ with these hyperplanes. Thus, we find that the foliation will always be smooth when $a\geq 0$, as $M$ is a spacelike hypersurface. Based on these observations, it is reasonable to call solutions of \eqref{aldkjf} \textit{timelike, spacelike, or null spacetime harmonic functions} depending on whether $a>0$, $a<0$, or $a=0$ respectively. Note that such null spacetime harmonic functions agree with the original definition of spacetime harmonic functions given at the beginning of the section.

As discussed in Section \ref{results}, null spacetime harmonic functions play an important role in a proof of the spacetime positive mass theorem. It should also be pointed out that timelike spacetime harmonic functions, in particular with $a=1$, have previously been derived in a different context and applied to another type of geometric inequality involving mass. Namely, in \cite{ChaKhuri1} (see also \cite{ChaKhuri2}) the authors generalize the Jang equation, utilized in the Schoen and Yau proof of the spacetime positive mass theorem \cite{SY2}, to be suitable for application to mass-angular momentum inequalities. In short, the Jang equation was generalized to include a lapse and shift, so that it can be used to recognize initial data from stationary spacetimes, such as Kerr. In this setting the initial data are taken to be axisymmetric, however as noted in Theorem 2.3 of \cite{ChaKhuri1} if the quantity $Y\equiv 0$ (vanishing shift) then the relevant conclusions hold without the symmetry hypothesis. If in addition the lapse is taken to be 1, the generalized Jang equation \cite[(2.37)]{ChaKhuri1} reduces to
\begin{equation}\label{00000}
\Delta f-\left(\mathrm{Tr}_g k\right)\sqrt{1+|\nabla f|^2}=0.
\end{equation}
This may be recognized as the generalized spacetime harmonic function equation with $a=1$, by replacing $k$ with $-k$. Note that the sign associated with $k$ is immaterial as it simply represents the direction (future or past) of the timelike normal used to compute the extrinsic curvature. This equation may be viewed in yet another way. Recall that the classical Jang equation was derived by taking the trace of the difference of second fundamental forms, namely, the second fundamental form $\left(1+|\nabla f|^2\right)^{-1/2}\nabla_{ij}f$ of the graph $t=f(x)$ in the product manifold $(\mathbb{R}\times M, dt^2 +g)$ and the given extrinsic curvature $k_{ij}$. Traditionally the trace was taken with respect to the induced metric $\bar{g}=g+df^2$ on the graph, whereas in \eqref{00000} the trace is taken with respect to $g$.

\subsection{Association with harmonic spinors}
\label{SS: spinor motivation}

Here we would like to draw a comparison between spacetime harmonic functions and spinorial approaches to mass is general relativity. To this end, consider the following characterization of the spacetime Hessian. Let $n$ denote a normal covector field to $M$ of length $-1$, and define a map
$T^*M\to T^* \mathbf{M}^4|_{M}$ by sending a covector $\alpha$ to the null covector $I(\alpha):=\alpha+|\alpha| n$. As above, the Levi-Civita connection on $\mathbf{M}^4$ will be denoted by $\pmb{\nabla}$. Then the tangential components of $\pmb{\nabla}I(du)$ of a spatial function $u$ coincide with the spacetime Hessian $\bar{\nabla}^2 u$.

In spinor-based approaches to the mass of 3-dimensional initial data sets \cite{ParkerTaubes,Witten}, one considers the bundle $\mathbb{S}$ of $SL_2(\mathbb{C})$ (Weyl) spinors associated to the restricted bundle $T\mathbf{M}^4|_{M}$. The spinor bundle $\mathbb{S}$ inherits two connections: an intrinsic connection $\nabla$ and one induced by $(\mathbf{M}^4,\mathbf{g})$. Using the later connection, one constructs the so-called Witten-Dirac operator $\bar{\slashed{\partial}}$. To analyze the mass of $(M,g,k)$, spinors solving $\bar{\slashed{\partial}}\psi=0$ and converging to a constant spinor at spatial infinity are studied. Clearly, there is an analogy one can make between this approach and the techniques described in the present paper -- the spinor $\psi$ corresponds to the 1-form $I(du)$, and the Witten-Dirac harmonic equation corresponds to the spacetime harmonic equation.

To make this relationship more concrete, observe that a given Weyl spinor $\psi$ can be `squared' to a null 1-form $\alpha$ defined by
\begin{equation}
\alpha(X)=\mathrm{Im}(X\cdot\psi,\psi)
\end{equation}
for any vector $X\in T\mathbf{M}^4$, where $(\;,\;)$ denotes the Hermitian pairing on $\mathbb{S}$ and $\cdot$ represents Clifford multiplication. A direct calculation reveals
\begin{equation}
\mathrm{Tr}_g(\pmb{\nabla} \alpha)=\mathrm{Im}(\bar{\slashed{\partial}}\psi,\psi).
\end{equation}
In particular, if $\bar{\slashed{\partial}}\psi=0$ and $\alpha=I(du)$ for a spatial  function $u$, then $u$ is spacetime harmonic. Informally, one can interpret a spacetime harmonic function as the integral of the square of a harmonic spinor. It is interesting to note that the level sets of $u$ are obscured from the perspective of $\psi$, yet they play a central role in our methods.

\section{The Spacetime Positive Mass Theorem}
\label{S: spacetime PMT}
\setcounter{equation}{0}
\setcounter{section}{6}

Consider the existence problem for asymptotically linear spacetime harmonic functions on asymptotically flat initial data. Although this is a nonlinear equation, the nonlinearity is relatively mild as it appears only in the first derivatives and is homogeneous of degree 1. For instance, the maximum principle still applies to the spacetime Laplace equation. Thus, we expect traditional existence results, mimicking those for linear equations with vanishing zeroth order term, but with a cap on the amount of regularity in general.

Let $(M,g,k)$ denote a smooth complete asymptotically flat initial data set with
boundary $\partial M$ (which may be empty), having a single end; the case of multiple ends may be treated similarly. Given a linear function $a_i x^i$ defined in the asymptotic end, with $\sum_{i}a_i^2 =1$, it is convenient to first construct an approximate solution in the asymptotic end that extends to all of $M$. More precisely,
by slightly generalizing \cite[Theorem 3.1]{Bartnik} we may solve the asymptotically linear Dirichlet problem for Poisson's equation
\begin{equation}
\Delta v=-\mathrm{Tr}_g k\quad\quad\text{ on }\quad\quad M,
\end{equation}
\begin{equation}
v=0\quad\text{ on }\quad \partial M,\quad\quad\quad v=a_i x^i +O_2(r^{1-q})\quad\text{ as }\quad r\rightarrow\infty,
\end{equation}
where $q>\frac{1}{2}$ is as in \eqref{AF}, $r=|x|$, and $O_2$ indicates in the usual way additional fall-off for each derivative taken up to order 2. Next let $M_r$ be a sequence of exhausting domains as $r\rightarrow\infty$, in which each member of the sequence consists of all points inside a large coordinate sphere $S_r$ in the asymptotic end. The Dirchlet problem for the spacetime harmonic function equation may be solved in $M_r$ via the Schauder fixed point theorem, with the solution $u_r$ prescribed to be $v$ on $S_r$ and a given fixed smooth function $c$ on $\partial M$. Uniform estimates for the difference $u_r -v$ may then be obtained by constructing a barrier function with the appropriate decay. By then passing to a limit, the desired solution $u\in C^{2,\alpha}(M)$ is found \cite[Section 4]{HKK} for the boundary value problem
\begin{equation}\label{7658}
\Delta u+\left(\mathrm{Tr}_g k\right)|\nabla u|=0\quad\quad\text{ on }\quad\quad  M,
\end{equation}
\begin{equation}\label{76890}
u=c\quad\text{ on }\quad \partial M,\quad\quad\quad\quad u=v +O_2(r^{1-2q})\quad\text{ as }\quad r\rightarrow\infty.
\end{equation}

In the next stage of the argument, the solution to \eqref{7658}, \eqref{76890} is taken on a generalized exterior region $M_{ext}$, with the inner boundary condition $c$ chosen to be appropriate constants on each component of the apparent horizon boundary. This solution is then inserted into the primary identity \eqref{master}, or rather a slightly expanded version of it \cite[Proposition 3.2]{HKK}, on the exhausting domain $M_r \subset M_{ext}$ to produce
\begin{equation}\label{2222}
\int_{\partial_{\neq 0} M_{r}}\left(\partial_{\upsilon}|\nabla u|+k(\nabla u,\upsilon)\right)dA\geq\int_{\underline{u}}^{\overline{u}}\int_{\Sigma_t}
\left(\frac{1}{2}\frac{|\bar{\nabla}^2 u|^2}{|\nabla u|^2}+\mu+J\left(\tfrac{\nabla u}{|\nabla u|}\right)-K\right)dA dt,
\end{equation}
where $\partial_{\neq 0}M_{r}$ is the open subset of $\partial M_r$ on which $|\nabla u|\neq 0$, $\upsilon$ is the unit outer normal, $\overline{u}$ and $\underline{u}$ denote the maximum and minimum values of $u$, and $\Sigma_t$ are $t$-level sets. The Gauss-Bonnet theorem may be used to replace the Gauss curvature $K$ with the Euler characteristic of level sets, and the geodesic curvature of the curves obtained by intersecting $\Sigma_t$ with the outer boundary $S_r$. Note that the inner boundary does not contribute a geodesic curvature term, since these boundary components are level sets of $u$. The outer boundary integral of \eqref{2222}, together with the Euler characteristic and geodesic curvature contributions, converges to the 4-momentum expression $8\pi\left(E+\langle\vec{a},P\rangle\right)$ as long as the Euler characteristics are not larger than 1. It should be noted that the original computation is carried out for coordinate cubes in the asymptotically flat end, although the corresponding limit agrees with the computation computed on coordinate spheres.
This Euler characteristic condition requires that there are no spherical level sets, and this follows from the strong maximum principle together with the topological condition $H_2(M_{ext},\partial M_{ext};\mathbb{Z})=0$. More precisely, the constants $c$ prescribed on the boundary are chosen in a particular manner to obtain two properties \cite[Lemma 5.1]{HKK}: 1) there exists a point $p$ in each boundary component such that $|\nabla u(p)|=0$, and 2) on each MOTS (MITS) boundary component $\partial_{\upsilon}u \leq (\geq) 0$. See Figure \ref{pic1}.

\begin{figure}[H]
\centering
\begin{picture}(0,0)
\put(-35,70){\large{$\partial M_{ext}$}}
\put(0,74){\vector(1,0){73}}
\put(230,120){{\color{blue}\large{$u^{-1}(c)$}}}
\put(230,75){{\color{red}\large{$u^{-1}(0)$}}}
\put(230,25){{\color{brass}\large{$u^{-1}(-1)$}}}
\end{picture}
\includegraphics[scale=.4]{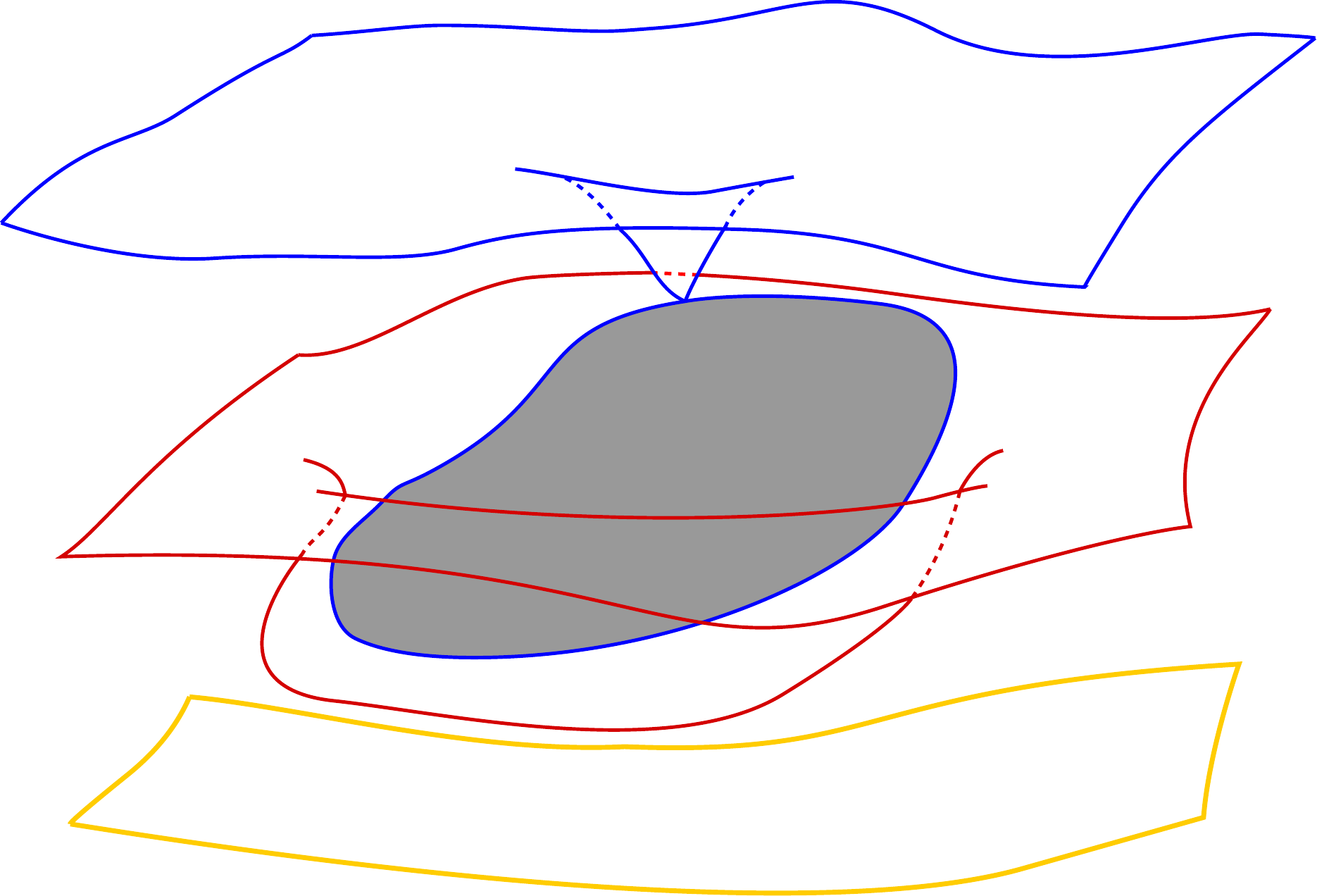}
\caption{Possible level sets of the spacetime harmonic function $u$ near components of the boundary $\partial M_{ext}$.}\label{pic1}
\end{figure}
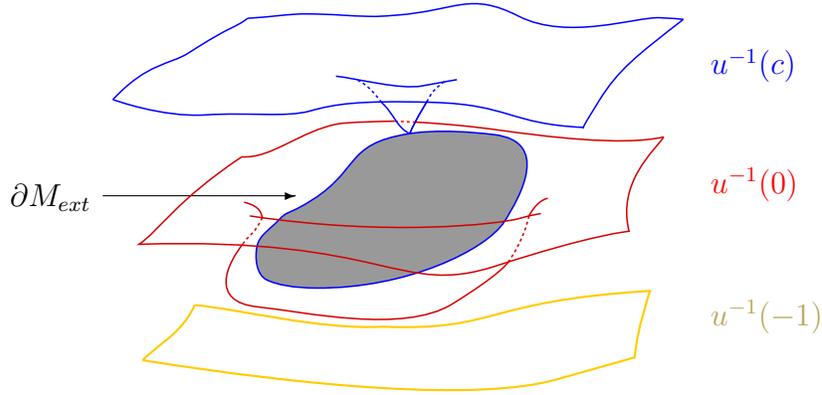

\noindent The first property implies, with help from the strong maximum principle, that no regular level sets of $u$ can be homologous to any part of the boundary, the second property guarantees that the inner boundary integral relates to the MOTS and MITS condition
\begin{equation}
\int_{\partial_{\neq 0}M_{ext}}\left(\partial_{\upsilon}|\nabla u|+k(\nabla u,\upsilon)\right)dA
=\sum_{j} \int_{\partial_j M_{ext}}\theta_{\pm}|\upsilon(u)|dA,
\end{equation}
where the sum is taken over the number of boundary components $\partial_j M_{ext}$, and the notation $\theta_{\pm}$ above indicates that the integrand contains $\theta_+$ for a MOTS
component and $\theta_-$ for a MITS component. It follows that the boundary integral vanishes, and the desired mass lower bound \eqref{spinteq} is attained.

In the above arguments the topological properties of the generalized exterior region play the important role of eliminating undesirable level set topologies. For each end $M_{end}$ of an asymptotically flat initial data set satisfying the dominant energy condition, there exists a corresponding generalized exterior region \cite[Proposition 2.1]{HKK}. To accomplish this
there are two primary steps. The first is to identify suitable (possibly immersed) MOTS and MITS to excise from $M$ in order to obtain a subset $M'\supset M_{end}$, whose compactification admits a metric of positive scalar curvature. This is established based on a reorganization of the arguments in \cite[Theorem 1.2]{ADGP}, and utilizes the dominant energy condition. The goal of the second step is then to reduce the first Betti number of $M'$ to zero by an iterative process that involves passing to finite sheeted covers. See Figure \ref{pic2}.

\begin{figure}[H]
\begin{picture}(0,0)
\put(5,50){\large{$(M,g,k)$}}
\put(305,50){\large{$(M_{ext},g_{ext},k_{ext})$}}
\end{picture}
\includegraphics[scale=.4]{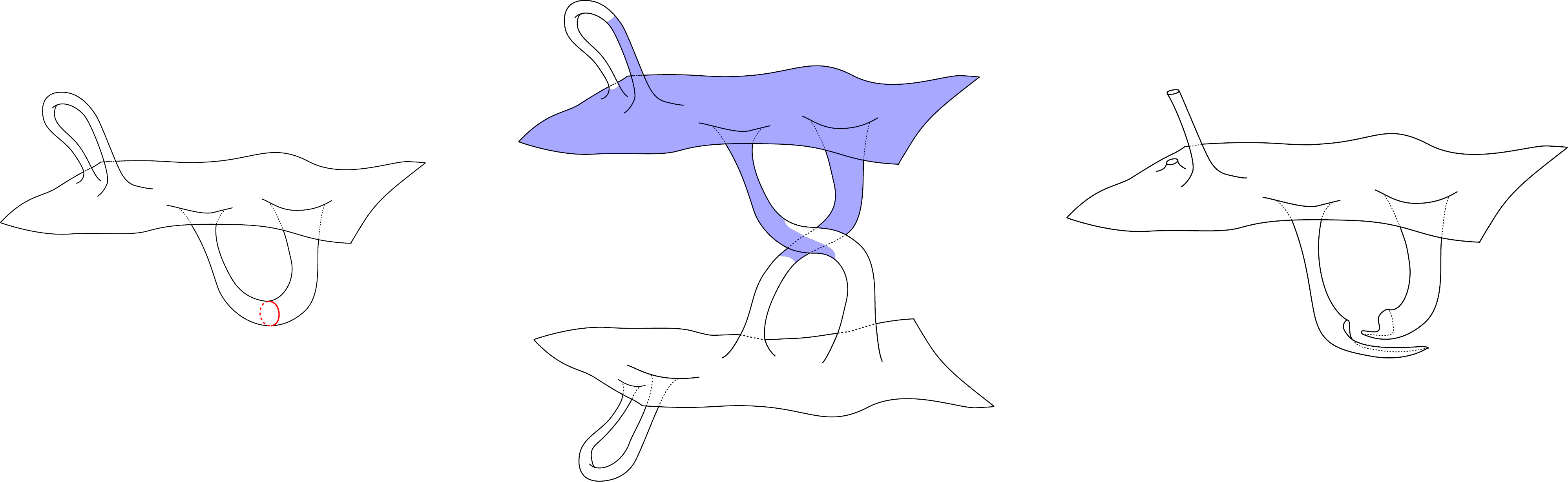}
\caption{A schematic description of the iterative covering space construction used to obtain a generalized exterior region. Starting with initial data possessing a nonvanishing first Betti number, pass to a double cover, and then remove an outermost MOTS to reduce the Betti number. This process may require a finite number of iterations to achieve vanishing first Betti number.}\label{pic2}
\end{figure}
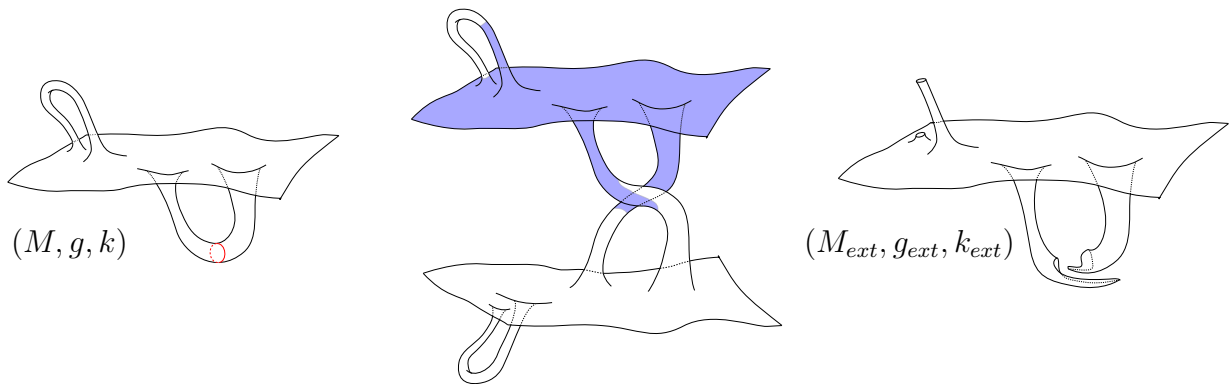

\section{The Hyperbolic Positive Mass Theorem}
\label{S: hyperbolic PMT}
\setcounter{equation}{0}
\setcounter{section}{7}

It is natural to expect that the methods utilizing spacetime harmonic functions to establish the asymptotically flat spacetime positive mass theorem, should be applicable in the asymptotically hyperboloidal setting as well. Indeed, the model situation in both cases involves a spacelike hypersurfaces of Minkowski space, one ending at spacelike infinity and the other at null infinity. However, there are several significant technical hurdles that are present in the hyperboloidal framework that are moot in the asymptotically flat scenario. First, the existence problem for the appropriate asymptotically linear spacetime harmonic functions is highly nontrivial in the asymptotically hyperboloidal case. A primary reason for this concerns the nonlinear term $\left(\mathrm{Tr}_g k\right)|\nabla u|$. Since $\mathrm{Tr}_g k \sim 3$ the nonlinearity is not nullified in the asymptotic end, as is the case in the asymptotically flat setting where $\mathrm{Tr}_g k =O(r^{-q-1})$, and therefore plays an important role in determining the asymptotic behavior of solutions. Secondly the boundary integrals arising from \eqref{2222} that converge to the mass are extremely difficult to compute, and require a precise expansion of the solution $u$. Nonetheless, this complications may be overcome leading to Theorem \ref{thm3}, which is established in \cite{BHKKZ}. Below we discuss in more detail various aspects of the proof.

\subsection{Spacetime harmonic functions in the asymptotically hyperboloidal setting}

In Section \ref{jang} we discussed how null linear functions of the coordinates in Minkowski space, such as $\ell=-t+\langle\vec{a},x\rangle$ with $|\vec{a}|=1$, restrict to spacelike hypersurfaces to yield canonical spacetime harmonic functions. In the asymptotically flat case, where slices are asymptotically totally geodesic, $\ell\sim\langle\vec{a},x\rangle$ in the asymptotic end suggesting that this be used as the asymptote for spacetime harmonic functions in general as in \eqref{76890}. In the asymptotically hyperboloidal case
\begin{equation}
\ell\sim-\sqrt{1+r^2}+\langle \vec{a},x\rangle=v_{\vec{a}},
\end{equation}
and so it is this function that is used as the asymptote for spacetime harmonic functions in this context. Given a smooth complete asymptotically hyperboloidal 3-dimensional initial set $(M,g,k)$, we prove in \cite{BHKKZ} that there exists $u\in C^{2,\alpha}(M)$ solving
\begin{equation}\label{=}
\Delta u+\left(\mathrm{Tr}_g k\right)|\nabla u|=0\quad\text{ on }\quad M,
\quad\quad u=v_{\vec{a}}+O_{2}(r^{-\tau}|v_{\vec{a}}|^{\tau})\quad\text{ as }\quad r\rightarrow\infty,
\end{equation}
where $\tau=\min\left(\frac{3}{2},\frac{q+1}{2}\right)$. As in the previous section, the equation is solved first on finite exhausting domains $M_r$, with Dirichlet condition $u=v_{\vec{a}}$ on the boundary sphere $S_r$, by a fixed point theorem. Then subconvergence to a global solution is obtained as $r\rightarrow\infty$, with the aid of barriers that are quite delicate to construct due to the rather asymmetrical nature of $v_{\vec{a}}$. The intricate structure of the barriers is  comparable with the complications faced when solving the Jang equation on asymptotically hyperbolic manifolds \cite{Sakovich}.

It turns out that the asymptotics for the solution as expressed in \eqref{=} are not sufficient to control the boundary integrals arising from the identity \eqref{master}. A more careful analysis is needed, and in particular a precise expansion of several orders is required, which informally is takes the form
\begin{equation}\label{eq:introexpansion}
u=v_{\vec{a}}+\frac{|v_{\vec{a}}|^\tau}{r^\tau}\left(A+\frac{B}{r^2}+O(r^{-4})\right)
\end{equation}
where $A$ and $B$ are functions defined on the spherical conformal infinity that are related to one another by a certain elliptic PDE. This expression may be compared to the expansion of a harmonic functions in spherical harmonics. Owing to the nature of the nonlinearities present in the spacetime Laplacian, achieving \eqref{eq:introexpansion} is challenging and technical. Although this expansion can be shown to correctly identify the mass from the boundary terms of \eqref{2222}, a direct approach is problematic. For this reason we pursue an alternative approach which utilizes the property that, in several respects, \eqref{eq:introexpansion} improves as the order of asymptotic hyperbolicity $q$ increases. To exploit this, we construct initial data sets $(\tilde{M},\tilde{g},\tilde{k})$ which {\emph{interpolate}} between the original initial data set $(M,g,k)$, and the model hyperboloid $(\mathbb{H}^3,b,b)$ near infinity. There is an annular interpolation region in $\tilde{M}$ where geometrical features of $(M,g,k)$ such as the dominant energy condition are severely disturbed. Nevertheless, we demonstrate a sense in which this region remembers the mass of the original initial data set up to an error which shrinks as the interpolation region escapes to infinity.

These observations allows us to carry out the following strategy to prove Theorem \ref{thm3}. Taking advantage of the improved expansion for spacetime harmonic functions on $(\tilde{M},\tilde{g},\tilde{k})$, it is possible to establish a version of inequality \eqref{spinteqh} for $(\tilde{M},\tilde{g},\tilde{k})$. In this version the boundary integrals converge to zero at infinity, since the asymptotic end is exactly hyperbolic space.
We then choose the transition annular regions in $(\tilde{M},\tilde{g},\tilde{k})$ to occur further and further out into the asymptotic end, and show that the these bulk transition integrals converge to the mass. Thus, the interpolation regions `encode the mass'. This will be discussed in more detail within the next subsection. It should be remarked that this style of argument seems to be uniquely available to the spacetime harmonic function approach, since it requires the ability to obtain `mass formulas' without the validity of the dominant energy condition as this is violated in the transition to hyperbolic space. Lastly we mention that, as before, the vanishing second homology hypothesis is used to control the Euler characteristic of level sets.

\subsection{Interpolation and a new concept of mass}\label{SS: interpolation}

In order to illustrate the concept of \textit{interpolation mass}, we will give the example of a Schwarzschild metric of mass $m>0$ which transitions to the flat metric. Let $\rho>m/2$, and consider a function $\mathbf{m}\in C^{\infty}(\mathbb{R})$ satisfying the following properties
\begin{equation}\label{-}
\mathbf{m}(r)=
\begin{cases}
m & r<\rho\\
0 & r>2\rho
\end{cases},
\quad\quad 0\leq\mathbf{m}(r)\leq m,\quad\quad \rho |\mathbf{m}'(r)|+\rho^2 |\mathbf{m}''(r)|\leq C,
\end{equation}
where $C$ is independent of $\rho$. This function may be used to define the Riemannian manifold $(\mathbb{R}^3 \setminus B_{m/2},g)$ which is Schwarzschild of mass $m$ for $r<\rho$, and Euclidean for $r>2\rho$, namely
\begin{equation}
g=  \left(1+\frac {\mathbf{m}(r)}{2r}\right)^4\delta.
\end{equation}
Observe that the scalar curvature takes the form
\begin{equation}\label{R formula}
R=
\begin{cases}
0 & r\in [m/2,\rho]\cup [2\rho,\infty)\\
\left(-\frac{4}{r}+O_{1}(r^{-2})\right) \mathbf{m}''(r) & r\in(\rho,2\rho)
\end{cases}.
\end{equation}
Let $u$ be a $g$-harmonic function with zero Neumann data on $S_{m/2}$, such that $u=x+O_{2}(1)$ as $r\rightarrow\infty$; this may be achieved by a slight modification of \cite[Theorem 3.1]{Bartnik} and the solution is unique up to addition of a constant. Then according to Theorem \ref{thm1}, and the fact that $g$ is of zero mass, we have
\begin{equation}\label{11111}
0\ge\int_{\mathbb{R}^3 \setminus B_{m/2}}\left(\frac{|\nabla^2 u|^2}{|\nabla u|}
+R|\nabla u|\right)dV\geq \int_{B_{2\rho}\setminus B_{\rho}}R|\nabla u|dV
+\int_{B_{\rho}\setminus B_{m/2}}\left(\frac{|\nabla^2 u|^2}{|\nabla u|}
+R|\nabla u|\right)dV.
\end{equation}
Treating $\rho$ as a parameter, note that the sequence of metrics $g$ is uniformly asymptotically flat in light of \eqref{-}, and thus there are uniform estimates such that $|\nabla u|=1+O_{1}(r^{-1})$ and $|S_r|=4\pi r^2+O_{1}(r)$. It follows that
\begin{equation}
\int_{B_{2\rho}\setminus B_{\rho}}R|\nabla u|dV
=\int_{\rho}^{2\rho}\mathbf{m}''(r)\left(-\frac{4}{r}+O_{1}(r^{-2})\right)
\left(4\pi r^2 +O_{1}(r)\right)dr=-16\pi m +O(\rho^{-1})
\end{equation}
as $\rho\to\infty$. Hence, all the mass is `stored' in the scalar curvature integral over an annulus near infinity, and the desired mass formula \eqref{intthm1} may be obtained by taking the limit in \eqref{11111}.

To summarize, the mass may be interpreted as measuring the amount of negative scalar curvature one has to `pay' to deform an asymptotically flat manifold into Euclidean space. In fact, this method of encoding the mass by interpolation works in the same manner in the non-time symmetric case, where the scalar curvature integral is replaced by a dominant energy condition expression, and it is valid in the asymptotically hyperboloidal setting as well where it is utilized in the proof of Theorem \ref{thm3}.

\subsection{Alternative approach: double interpolation}\label{SS:alternative}

In the previous discussion there was a single interpolation, either to Euclidean or hyperbolic space, depending on whether the initial data are asymptotically flat or asymptotically hyperboloidal. However, we would like to  point out here an alternative approach that requires two interpolations. Start with an asymptotically hyperboloidal initial data set $(M,g,k)$. Then perform the first interpolation to achieve the initial data of the hyperboloid in Minkowski space outside a large coordinate sphere, and denote this newly deformed data by $(M',g',k')$. This will of course destroy the dominant energy condition within an annulus $A'$ in the asymptotic end, where the mass will be stored. In the second step, interpolate to Euclidean 3-space further out in the asymptotic end to obtain $(M'',g'',k'')$. Thus, there are inner and outer annuli $A'$ and $A''$, in which $A'$ encodes the mass but fails the dominant energy condition, while $A''$ contains no mass but satisfies the dominant energy condition.  Out side of $A'$ the data agrees with $(\mathbb{H}^3,b,b)$, while outside of $A''$ the data agrees with $(\mathbb{R}^3,\delta,0)$. See Figure \ref{pic11}. Note that in the second deformation the dominant energy condition is preserved, since the deformation can be carried explicitly in terms of an embedded spacelike hypersurfaces in Minkowski space. The new data $(M'',g'',k'')$ is now asymptotically flat, and hence we may apply the existence result for spacetime harmonic functions \cite{HKK} to obtain a solution with linear asymptotics. This function may then be used in the mass formula \eqref{spinteq}, following arguments similar to those in Section \ref{SS: interpolation}, to establish the hyperbolic version of the positive mass theorem by using the asymptotically flat result. The advantage of this process is that it avoids the complicated computations in the asymptotically hyperbolic end, in which it is shown that the boundary integrals from the mass formula converge to the appropriate quantity. On the other hand, it is more difficult to construct barriers for the spacetime harmonic function on $(M'',g'',k'')$, which are independent of the radial parameter determining the location of the two annuli $A'$ and $A''$. For this reason, in \cite{BHKKZ} this strategy is not pursued. Nonetheless, using the interpolation method we are able to prove a special case of the rigidity statement for the hyperbolic positive mass theorem. See for instance \cite{HJM}, where the full result is established.

\begin{figure}[H]
\centering
\begin{picture}(0,0)
\put(133,45){\large{$M$}}
\put(85,45){\large{$A'$}}
\put(50,45){\large{$\mathbb H^3$}}
\put(18,45){\large{$A''$}}
\put(-20,45){\large{$\R^3$}}
\end{picture}
\includegraphics[scale=.6]{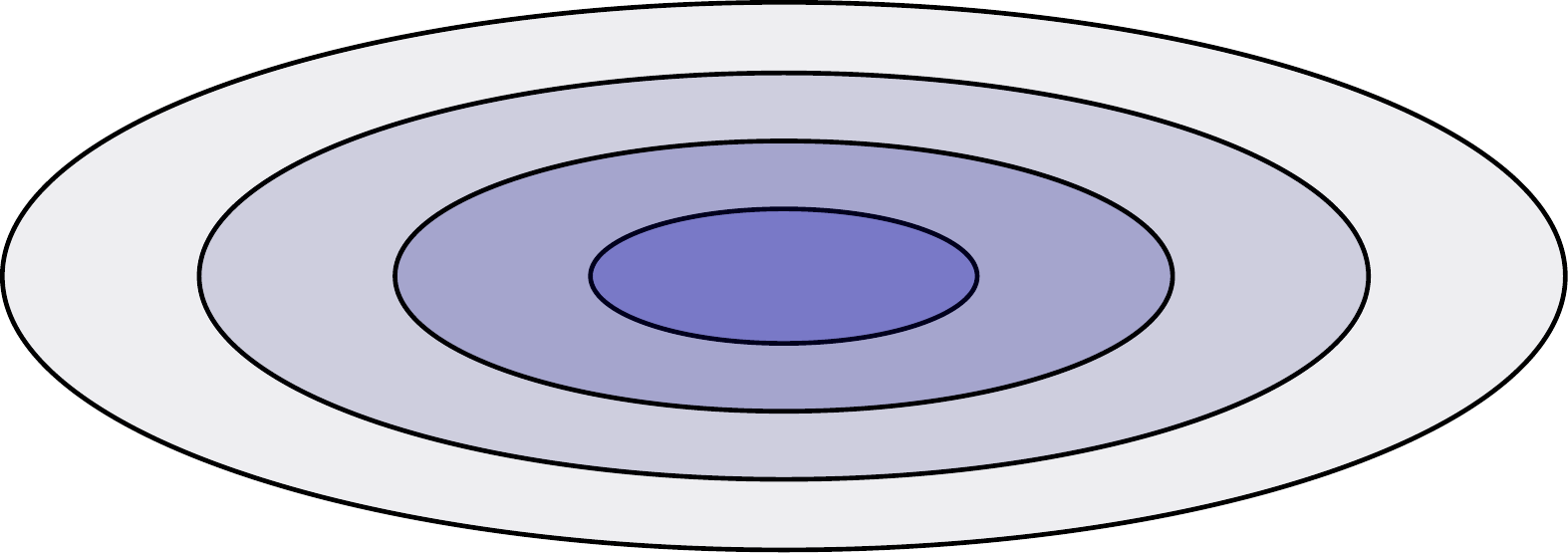}
\caption{Schematic of the double interpolation.}\label{pic11}
\end{figure}
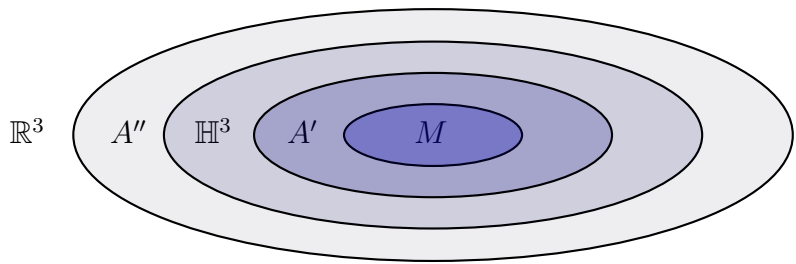

\begin{proposition}
Let $(M,g)$ be a complete Riemannian 3-manifold with $R\geq -6$. If, outside a compact set, this manifold is isometric to the compliment of a ball hyperbolic space, then $(M,g)$ is globally isometric to $\mathbb{H}^3$.
\end{proposition}

\begin{proof}
The initial data $(M,g,g)$ may be viewed as the end result of a first interpolation, since it agrees with the standard hyperboloid outside a compact set. As described above, there is then an interpolation $(M'',g'',k'')$ between $(M,g,g)$ and Euclidean data, which in this case satisfies the dominant energy condition globally. Since this new manifold has zero ADM mass, the spacetime positive mass theorem, Theorem \ref{thm2}, then applies to show that $(M'',g'',k'')$ arises from an embedding into Minkowski space. Since the interpolated data agrees with the original into its hyperbolic end, there is a portion $(M_r,g,g)$ lying inside a large coordinate sphere which has an umbilic embedding in Minkowski space.

We claim that the image of the isometric embedding $\mathcal{I}:(M_r,g)\hookrightarrow\mathbb{M}^4$ must be a portion of the standard hyperboloid. To see this consider the `center function' $\mathcal{C}:M_r\to \mathbb{M}^4$ given by
\begin{equation}
\mathcal{C}(x)=\mathcal{I}(x)-n\left(\mathcal{I}(x)\right),
\end{equation}
where $n$ is the unit timelike normal to $\mathcal{I}(M_r)$ and $x=(x^1,x^2,x^3)$ are local coordinates. Let $T_i =\partial_i \mathcal{I}(x)$ denote the coordinate vector fields spanning the tangent space along the embedding. Then the umbilicity condition yields
\begin{equation}
\partial_i \mathcal{C}(x)=T_i-k_{i}^{j}T_j=T_i -g_i^j T_j=0,\quad\quad i=1,2,3.
\end{equation}
It follows that $\mathcal{C}(x)=\mathcal{C}_0$ is a constant map, and in fact
\begin{equation}
|\mathcal{I}(x)-\mathcal{C}_0|^2 =-1,
\end{equation}
so that $\mathcal{I}(M_r)$ lies in a Lorentzian sphere, or rather the standard hyperboloid. We may now conclude that $(M,g)$ is isometric to $\mathbb{H}^3$.
\end{proof}

\section{A Positive Mass Theorem With Charge}
\label{S: charge}
\setcounter{equation}{0}
\setcounter{section}{8}

In this section we establish Theorem \ref{thm4}. Let $(M,g,\mathcal{E})$ be a smooth complete charged asymptotically flat initial data set, having divergence free electric field $\mathcal{E}$, mass $m$, and total charge $Q$. In particular, $M$ is diffeomorphic to $\mathbb{R}^3 \setminus \cup_{i=1}^I p_i$ with a single asymptotically flat end $M_{end}$ and $I$ asymptotically cylindrical ends represented by $p_i$. We say that a subset of $M$ is an \textit{asymptotically cylindrical end} if there is a diffeomorphism
$\psi:M_{cyl}\rightarrow [1, \infty)\times S^2$ such that
\begin{equation}\label{acylindrical}
|^{g_0}\nabla^{l}\left(\psi_{*} g -g_0\right)|_{g_0}+|\psi_* \mathcal{E} - q_0\partial_s |_{g_0}=O(s^{-1})\quad\text{ as }\quad s\rightarrow\infty
\end{equation}
for $l=0,1,2$, where $g_0 = ds^2 +\sigma_0$ with $\sigma_0$ a metric on $S^2$, and $q_0 \neq 0$ is a constant that determines the charge of the end. This definition is modeled on the asymptotically cylindrical ends present in a time slice of the Majumdar-Papapetrou spacetime.

\subsection{The basic integral identity.}

The first step is to establish the primary integral identity on compact subdomains $\Omega\subset M$. Assume that $u$ satisfies the charged Laplacian equation
\begin{equation}\label{=--}
\Delta u-\langle\mathcal{E},\nabla u\rangle=0 \quad\text{ on }\quad \Omega,
\end{equation}
and recall the charged Hessian
\begin{equation}
\hat{\nabla}_{ij}u=\nabla_{ij}u+\mathcal{E}_i u_j+\mathcal{E}_j u_i
-\langle \mathcal{E},\nabla u\rangle g_{ij}
\end{equation}
which satisfies
\begin{equation}
|\nabla^2u|^2=|\hat\nabla^2u|^2-2|\mathcal{E}|^2|\nabla u|^2+\langle \mathcal{E},\nabla u\rangle^2-4\mathcal{E}^i u^j\nabla_{ij}u.
\end{equation}
Inserting equation \eqref{=--} into the main identity \eqref{7890}, assuming that $|\nabla u|\neq 0$, produces
\begin{equation}
\div\left(\nabla |\nabla u|-\Delta u\frac{\nabla u}{|\nabla u|}\right)
=\frac{1}{2|\nabla u|}\left(|\hat\nabla^2 u|^2-4\mathcal{E}^i u^j \nabla_{ij}u +|\nabla u|^2(R-2|\mathcal{E}|^2-2K)\right).
\end{equation}
Next observe that
\begin{align}
\begin{split}
-\frac1{|\nabla u|}\mathcal{E}^i u^j \nabla_{ij}u=&-\div\left(|\nabla u| \mathcal{E}\right)+|\nabla u|\mathrm{div}\mathcal{E}
+\frac1{|\nabla u|}\mathcal{E}^i u^j \nabla_{ij}u-\mathcal{E}^i\partial_i|\nabla u|\\
=&-\div\left(|\nabla u|\mathcal{E}\right)+|\nabla u|\mathrm{div}\mathcal{E}.
\end{split}
\end{align}
Combining the last two equations, using the divergence free property of the electric field, and integrating over $\Omega$ yields
\begin{equation}\label{======}
\int_{\partial \Omega}\left(\partial_\upsilon |\nabla u|-\Delta u\frac{\nabla_\upsilon u}{|\nabla u|}+2|\nabla u|\langle \mathcal{E},\upsilon\rangle\right)dA
\geq\frac{1}{2}\int_{\Omega}\left(\frac{|\hat\nabla^2u|^2}{|\nabla u|}+(R-2|\mathcal{E}|^2-2K)|\nabla u|\right)dV,
\end{equation}
where $\upsilon$ is the unit outer normal to $\partial \Omega$. The inequality, rather than equality, arises from the procedure that passes from the case $|\nabla u|\neq 0$ to the general setting, see \cite[Section 3]{HKK}.

\subsection{Existence of charged harmonic functions and the mass formula}

Consider a 2-parameter family of exhausting domains $M_{r,\epsilon}\subset M$, which are defined to consist of all points lying between a large coordinate sphere $S_r$ in the asymptotically flat end and the collection of coordinate ($s=\epsilon^{-1}$ cross-section) spheres $S_{\epsilon}$ in the asymptotically cylindrical ends. In order to better control level set topology, we will cap-off the spheres $S_{\epsilon}$ with a collection of 3-balls $\Omega_{\epsilon}$ to achieve a new manifold $\tilde{M}_{r,\epsilon}$ that is diffeomorphic to a 3-ball. The metric on $\tilde{M}_{r,\epsilon}$ is defined to be a smooth extension of $g$ such that the curvature of $g_{\epsilon}$, the restriction of the extended metric to $\Omega_{\epsilon}$, remains uniformly bounded. In addition, the electric field is extended trivially so that $\mathcal{E}=0$ on $\Omega_{\epsilon}$. This unrefined extension of the geometry and electric field will of course destroy the charged dominant energy condition $R\geq 2|\mathcal{E}|^2$ on $\Omega_{\epsilon}$, however it will be sufficient for our purposes.

We first construct a model function $v$ to which the desired charged harmonic function should asymptote. Let $x$ be one of the Cartesian coordinates in $M_{end}$. Since $\langle\mathcal{E},\partial_x\rangle\in C_{-q-1}^{0,\alpha}(M_{end})$, it follows from \cite[Theorem 3.1]{Bartnik} that there is a solution of the Dirichlet problem
\begin{equation}
\Delta v=\langle\mathcal{E},\partial_x\rangle\quad\quad\text{ on }\quad\quad \tilde{M}_{\infty,\epsilon}\setminus \tilde{M}_{r_0,\epsilon},
\end{equation}
\begin{equation}
v=0\quad\text{ on }\quad S_{r_0},\quad\quad\quad\quad v=x+O_2(r^{1-q})\quad\text{ as }\quad r\rightarrow\infty,
\end{equation}
where $q$ is as in \eqref{AF} and $r_0$ is a large fixed radius. The function $v$ may be extended smoothly so that it is defined on all of $\tilde{M}_{\infty,\epsilon}$, agrees with the solution above for $r>r_0$, and vanishes for $r<r_0 /2$. Now solve the Dirichlet problem
\begin{equation}
\Delta w_{r,\epsilon}-\langle\mathcal{E},\nabla w_{r,\epsilon}\rangle=f\quad\text{ on }\quad\tilde{M}_{r,\epsilon},\quad\quad\quad\quad w_{r,\epsilon}=0\quad\text{ on }\quad S_r,
\end{equation}
for $r>r_0$ where
\begin{equation}\label{787878}
f:=-\Delta v+\langle \mathcal{E},\nabla v\rangle=-\langle\mathcal{E},\partial_x\rangle+\langle \mathcal{E}, \partial_x+O(r^{-q})\rangle=O(r^{-2q-1}).
\end{equation}
Note that since $\mathcal{E}$ is discontinuous at $S_{\epsilon}$, the solution is only guaranteed to be $C^{1,\alpha}$ in a neighborhood of this surface, although it is smooth everywhere else. By setting $u_{r,\epsilon}=v+w_{r,\epsilon}$, we obtain a solution of \eqref{=--} on $\tilde{M}_{r,\epsilon}$ with $u_{r,\epsilon}=v$ on $S_r$, and uniform estimates will show that this sequence converges to the desired charged harmonic function.

Uniform $C^0$ bounds for $w_{r,\epsilon}$ may be established by constructing a global barrier. More precisely, in light of \eqref{787878} the appropriate uniform sub and super solutions may be found as in \cite[Section 4.2]{HKK} with the decay rate $r^{1-2q}$ in $M_{end}$. It follows from elliptic theory that there are uniform $W^{2,p}$ estimates on compact subsets, with control of higher order derivatives away from $S_{\epsilon}$, and thus $w_{r,\epsilon}$ subconverges to a solution $w_{\epsilon}$. We then have $u_{\epsilon}=v+w_{\epsilon}\in C^{1,\alpha}$, and smooth away from $S_{\epsilon}$, weakly satisfying
\begin{equation}\label{uuui}
\Delta u_{\epsilon}-\langle\mathcal{E},\nabla u_{\epsilon}\rangle=0\quad\text{ on }\quad\tilde{M}_{\infty,\epsilon},\quad\quad\quad u_{\epsilon}=v+O_{2}(r^{1-2q})\quad\text{ as }\quad r\rightarrow\infty.
\end{equation}
The barriers imply that $u_{\epsilon}$ remains uniformly bounded independent of $\epsilon$ along each capped-off cylindrical end. Thus, employing a separation of variables argument along each such end shows that for large $s$ we have
\begin{equation}
u_{\epsilon}\sim c_0+\sum_{j=1}^{\infty} c_j e^{\left(\frac{q_0 -\sqrt{q_0^2 +4\lambda_j}}{2}\right)s}\chi_{j},
\end{equation}
for some constants $c_j$ where $\lambda_j> 0$, $\chi_j$ are the eigenvalues and eigenfunctions of the Laplacian on $(S^2,\sigma_0)$. It follows that $u_{\epsilon}$ is nearly constant on $S_{\epsilon}$, and since $u_{\epsilon}$ is harmonic on $\Omega_{\epsilon}$ we find that
\begin{equation}\label{0-=}
\sup_{\Omega_{\epsilon}}|\nabla u_{\epsilon}|\rightarrow 0\quad\text{ as }\quad \epsilon\rightarrow 0.
\end{equation}

We are now in a position to establish the mass formula. In particular, since $\tilde{M}_{\infty,\epsilon}\cong\mathbb{R}^3$ has trivial topology, no component of a regular level set can be compact as otherwise it would bound a compact domain leading to a contradiction via the maximum principle. In analogy with \cite{HKK}, we may then
apply \eqref{======} separately on both components of $\tilde{M}_{\infty,\epsilon}\setminus S_{\epsilon}$, and add the resulting inequalities together, to obtain
\begin{align}\label{lkihgt}
\begin{split}
8\pi\left(m-|Q|\right)\geq &\frac{1}{2}\int_{\tilde{M}_{\infty,\epsilon}\setminus\Omega_{\epsilon}}
\left(\frac{|\hat{\nabla}^2 u_{\epsilon}|^2}{|\nabla u_{\epsilon}|}+(R-2|\mathcal{E}|^2)|\nabla u_{\epsilon}|\right)dV\\
&+\frac{1}{2}\int_{\Omega_{\epsilon}}R_{\epsilon}|\nabla u_{\epsilon}|dV_{\epsilon}+\int_{S_{\epsilon}}2|\nabla u_{\epsilon}|\langle\mathcal{E},\upsilon\rangle dA,
\end{split}
\end{align}
where $\upsilon$ is the unit normal to $S_{\epsilon}$ pointing to the asymptotically flat end, and $R_{\epsilon}$, $dV_{\epsilon}$ are the scalar curvature and volume form of $g_{\epsilon}$. It should be noted that the boundary term $\partial_{\upsilon}|\nabla u_{\epsilon}|$, at $S_{\epsilon}$, might appear to cause difficulty due to the fact that $u_{\epsilon}$ is only  $C^{1,\alpha}$ at across this surface. However, precisely this scenario has been treated in
\cite{HMT}, and is shown not to contribute adversely to the inequality. The estimate \eqref{0-=}, together with the control imposed on the geometry of $\Omega_{\epsilon}$, guarantees that the two integrals on the second line of \eqref{lkihgt} converge to zero.
Finally, the uniform estimates for $u_{\epsilon}$ on compact subsets shows that this sequence of functions subconverges to a solution $u$ of the charged Laplacian \eqref{=--} on $M$, which is bounded along the asymptotically cylindrical ends and satisfies the asymptotics of \eqref{uuui} in $M_{end}$. Moreover, \eqref{lkihgt} and Fatou's lemma imply that
\begin{equation}\label{frtyh}
m-|Q|\geq\frac{1}{16\pi}\int_{M}
\left(\frac{|\hat{\nabla}^2 u|^2}{|\nabla u|}+(R-2|\mathcal{E}|^2)|\nabla u|\right)dV,
\end{equation}
which gives the desired inequality of Theorem \ref{thm4}.

\subsection{The case of equality}

When the charged dominant energy condition $R\geq 2|\mathcal{E}|^2$ holds, \eqref{frtyh} implies that $m\geq |Q|$. Here we consider the rigidity phenomena when $m=|Q|$. This immediately implies that
\begin{equation}
R=2|\mathcal{E}|^2,\quad\quad
\nabla_{ij}u+\mathcal{E}_i u_j+\mathcal{E}_j u_i
-\langle \mathcal{E},\nabla u\rangle g_{ij}=0.
\end{equation}
Taking a divergence of the vanishing Hessian equation produces
\begin{equation}\label{67ty}
u^i R_{ij}-|\mathcal{E}|^2 u_j + \langle\mathcal{E},\nabla u\rangle \mathcal{E}_j +u^i\nabla_i \mathcal{E}_j=0.
\end{equation}
Note that $u$ may be chosen to asymptote to any one of the asymptotically
flat coordinate functions in $M_{end}$, with the validity of \eqref{67ty} unchanged. We will denote the three charged harmonic functions that asymptote to $x$, $y$, and $z$ by $u^{x}$, $u^y$, and $u^z$. If $\nabla u^x$, $\nabla u^y$, and $\nabla u^z$ are linearly independent at each point, it will
follow that
\begin{equation}\label{riccie}
R_{ij}-|\mathcal{E}|^2 g_{ij}+\mathcal{E}_i \mathcal{E}_j
+\nabla_i \mathcal{E}_j=0.
\end{equation}

To see that, indeed, these three gradient vector fields always form a basis for the tangent space, assume that there is a point $a\in M$ where this is fails. Then there are constants $c_x$, $c_y$, and $c_z$, not all zero, such that
\begin{equation}
V=c_x\nabla u^x+c_y\nabla u^y+c_z\nabla u^z
\end{equation}
vanishes at $a$. Let $b\in M_{end}$ be a point sufficiently far out in the asymptotic end, and connect $a$ to $b$ by a geodesic $\gamma$. Observe that
\begin{align}
\begin{split}
\nabla_{\dot\gamma}V_{i}
=&c_x \left(-\nabla_{\dot \gamma}u^x \mathcal{E}_i-\nabla_i u^x \langle \mathcal{E},\dot \gamma\rangle +\langle \mathcal{E},\nabla u^x\rangle \langle\partial_i, \dot \gamma\rangle \right)\\
&+c_y \left(-\nabla_{\dot \gamma}u^y \mathcal{E}_i-\nabla_i u^y \langle \mathcal{E},\dot \gamma\rangle+\langle \mathcal{E},\nabla u^y\rangle \langle\partial_i, \dot \gamma\rangle \right)\\
&+c_z \left(-\nabla_{\dot \gamma}u^z \mathcal{E}_i-\nabla_iu^z \langle \mathcal{E},\dot \gamma\rangle+\langle \mathcal{E},\nabla u^z\rangle
\langle\partial_i, \dot \gamma\rangle \right).
\end{split}
\end{align}
Therefore if $e_1$, $e_2$ are parallel fields along $\gamma$, such that $(e_1,e_2,\dot\gamma)$ forms an orthonormal basis, then
\begin{align}
\begin{split}
\nabla_{\dot\gamma}\langle V,\dot\gamma\rangle=&-\langle\mathcal{E},\dot\gamma\rangle \langle V,\dot\gamma\rangle+\langle\mathcal{E},e_1\rangle \langle V,e_1\rangle
+\langle\mathcal{E},e_2\rangle \langle V,e_2\rangle,\\
\nabla_{\dot\gamma}\langle V,e_i \rangle=&-\langle\mathcal{E},e_i\rangle
\langle V,\dot \gamma\rangle-\langle\mathcal{E},\dot\gamma\rangle \langle V,e_i\rangle,\quad\quad i=1,2.
\end{split}
\end{align}
This is a first order homogeneous linear system of ODEs for $V$ along $\gamma$. Since $|V(a)|=0$ it follows that $V$ vanishes along $\gamma$, and in particular $|V(b)|=0$. However, this contradicts the fact that the three gradient fields are linearly independent in the asymptotically flat end. We conclude that \eqref{riccie} is verified.

Next, notice that a consequence of \eqref{riccie} is that $\nabla \mathcal{E}$ is symmetric, and hence $\mathcal{E}$ is closed as a 1-form. Since the first cohomology $H^1(M)$ is trivial, there exists a globally defined function $h$ such that $\mathcal{E}=dh$. The function $h$ is harmonic, as $\mathcal{E}$ is divergence free, and according to the decay assumed on the electric field and the metric it follows that, after possibly adding a constant, $h=c/r+O(r^{-2q})$ for some constant $c$. Therefore, the metric $\tilde{g}=e^{-2h}g$ is asymptotically flat in $M_{end}$. Moreover, with the help of \eqref{riccie} we find that it is also flat
\begin{equation}
\tilde R_{ij}=R_{ij}+\nabla_{ij}h+h_i h_j+\left(\Delta h-|\nabla h|^2 \right)g_{ij}=\left(\Delta h \right)g_{ij}=0.
\end{equation}
Consider now the behavior of $h$ along an asymptotically cylindrical end. The decay along this end \eqref{acylindrical} yields $h\sim q_0 s$, and therefore each asymptotically cylindrical end is either conformally closed or opened in $\tilde{g}$ depending on whether $q_0$ is positive or negative. This may be seen more explicitly by making the change of radial coordinate $r=e^{-q_0 s}$. In the case that the end is conformally closed, $\tilde{g}$ is asymptotically conical and flat, thus the removable singularity theorem of \cite[Theorem 3.1]{SmithYang} shows that $\tilde{g}$ smooth extends across the singular point. In the case that there are conformally opened ends we arrive at a contradiction, as such an end would give a positive contribution to the mass of $M_{end}$, yet since $\tilde{g}$ is flat the mass of $(M_{end},\tilde{g})$ is zero. In particular, each $q_0> 0$ and so the electric charge associated with each asymptotically cylindrical end is positive. Furthermore, after adding a point for each asymptotically cylindrical end the metric $\tilde{g}$ is complete, so that $(M,\tilde{g})\cong (\mathbb{R}^3 \setminus \cup_{i=1}^I p_i , \delta)$. We then have $g=\phi^2 \delta$ where $\phi=e^h$. Finally, a computation shows that $\Delta_{\delta}\phi=0$ and $\phi=O(r_i^{-1})$ where $r_i$ is the Euclidean distance to $p_i$. Since in addition $\phi\rightarrow 1$ in the asymptotically flat end, the removable singularity theorem for harmonic functions combined with the maximum principle shows that
\begin{equation}
\phi=1+\sum_{i=1}^{I}\frac{q_i}{r_i},
\end{equation}
for some positive constants $q_i$. The original initial data $(M,g,\mathcal{E})$ is then isometric to the time slice of a Majumdar-Papapetrou spacetime.

\section{Some Open Questions}
\label{S: questions}
\setcounter{equation}{0}
\setcounter{section}{9}

\subsection{Extension to higher dimensions}

One of the primary unresolved issues surrounding the level set techniques discussed here, is whether they can be generalized to higher dimensions in a meaningful way. Although independent workarounds for dealing with the singularities of stable minimal hypersurfaces
have been found \cite{Lo1,SY3}, and the spinorial techniques are available in all dimensions for spin manifolds, it would be of significant interest for the study of scalar curvature to extend the relatively simple harmonic level set approach to obtain an alternate proof of the positive mass theorem in higher dimensions. One of the primary difficulties is the reliance, in three dimensions, on the Gauss-Bonnet theorem for regular level sets. In higher dimensions, the total scalar curvature integral appears in its place but is not related to topological information. Perhaps some form of dimensional reduction can be carried in this context, similar to the strategy of Schoen and Yau \cite{SY}.

\subsection{Spacetime Penrose inequality}

Apparent horizons within initial data should contribute to the ADM mass. The precise relationship expressing this conjecture is known as the Penrose inequality. Namely, if the dominant energy condition is satisfied then the mass should satisfy the lower bound $m\geq\sqrt{\mathcal{A}/16\pi}$, where $\mathcal{A}$ is the minimal area required to enclose the outermost apparent horizon, and equality should be achieved only for slices of the Schwarzschild black hole. While this statement has been confirmed \cite{Bray1,HI} in the time symmetric case $k=0$, it remains open in general. In this article we have seen how to extend level set techniques for the Riemannian positive mass theorem to the spacetime setting, using spacetime harmonic functions. Since the proof of the Riemannian Penrose inequality given by Huisken and Ilmanen \cite{HI} is based on a level set characterization of inverse mean curvature flow, it remains an intriguing possibility that a similar extension may be possible that is based on the level sets of some as yet undiscovered `spacetime inverse mean curvature flow'. A potential candidate for this purpose is the inverse null mean curvature flow studied by Moore \cite{Moore}.

\subsection{Other geometric inequalities}

We have seen in this article how choosing different expressions for $\Delta u$, in combination with the primary identity \eqref{master}, has led to several mass related inequalities. Surely there are more choices for equations to be found, with further applications. In particular, with each model stationary electro-vacuum black hole solution, there is a corresponding conjectured geometric inequality relating the total mass to the area of horizons, angular momentum, and charge of the black holes contained within  initial data. These relations are ultimately motivated by the grand cosmic censorship conjecture \cite{Mars,Penrose}. The classical Penrose inequality discussed above is one member of this family. Very little is known for several of these inequalities, including the Penrose inequality with angular momentum and the hyperbolic Penrose inequality, of which there are two versions depending on whether the apparent horizon is a minimal surface or a surface of constant mean curvature $H=2$. Is it possible that the level set techniques discussed here can be developed further to address this larger class of inequalities?

\subsection{Some enigmatic features}

There are a few aspects of the level set method with spacetime harmonic functions that are not quite well-understood. As discussed in Section \ref{S: spacetime harmonic functions}, the spacetime harmonic function equation naturally fits into a larger class of equations that are broken up into three groups, depending on the causal character of the hyperplanes in Minkowski space which generate the canonical solutions. The functions used in the proof of the spacetime and hyperboloidal positive mass theorem arise from the `null' class of such equations. While the `spacelike' version seems to be undesirable due to the possibility of producing irregular slicings even for initial data within Minkowski space, it is not clear why the `timelike' spacetime harmonic functions should not be as useful as their null counterparts. 

Another related aspect that deserves further investigation is the choice of asymptote for the spacetime harmonic functions. In both the asymptotically flat and asymptotically hyperboloidal settings, the asymptotes of spacetime harmonic functions have gradients that are either Killing or more generally conformal Killing fields with respect to the background model geometry. 
It would be of interest to determine what applications may arise by utilizing the other Killing/conformal Killing fields as asymptotes for the gradient of spacetime harmonic functions.

\end{document}